\documentclass{amsart}
\usepackage{amsfonts}
\usepackage{hyperref}
\usepackage{amssymb}

\newtheorem{theorem}{Theorem}
\theoremstyle{plain}

\newtheorem{corollary}{Corollary}

\newtheorem{definition}{Definition}

\newtheorem{lemma}{Lemma}

\newtheorem{proposition}{Proposition}
\newtheorem{remark}{Remark}

\numberwithin{equation}{section}
\input{tcilatex}

\begin{document}
\title[Cusa type inequalities for trigonometric functions]{Sharp Cusa type
inequalities for trigonometric functions with two parameters}
\author{Zhen-Hang Yang}
\address{Power Supply Service Center, Zhejiang Electric Power Corporation
Research Institute, Hangzhou City, Zhejiang Province, 310009, China}
\email{yzhkm@163.com}
\date{March 31, 2014}
\thanks{This paper is in final form and no version of it will be submitted
for publication elsewhere.}

\begin{abstract}
Let $\left( p,q\right) \mapsto \beta \left( p,q\right) $ be a function
defined on $\mathbb{R}^{2}$. We determine the best or better $p,q$ such that
the inequality%
\begin{equation*}
\left( \frac{\sin x}{x}\right) ^{p}<\left( >\right) 1-\beta \left(
p,q\right) +\beta \left( p,q\right) \cos ^{q}x
\end{equation*}%
holds for $x\in \left( 0,\pi /2\right) $, and obtain a lot of new and sharp
Cusa type inequalities for trigonometric functions. As applications, some
new Shafer-Fink type and Carlson type inequalities for arc sine and arc
cosine functions, and new inequalities for trigonometric means are
established.
\end{abstract}

\maketitle

\section{Introduction}

For $x\in \left( 0,\pi /2\right) $, the double inequality%
\begin{equation}
\left( \cos x\right) ^{1/3}<\frac{\sin x}{x}<\frac{2+\cos x}{3}  \label{M-C}
\end{equation}%
holds true, where the left inequality was obtained by Adamovi\'{c} and
Mitrinovi\'{c} (see \cite[2, p. 238]{Mitrinovic.AI.1970}), while the right
one is due to Cusa and Huygens (see, e.g., \cite{Huygens}) and it is now
known as \emph{Cusa's inequality} \cite{Chen.JIA.2011.136}, \cite%
{Mortitc.MIA.14(2011)}, \cite{Neuman.13(4)(2010)}, \cite%
{Sandor.RGMIA.8(3)(2005)}, \cite{Zhu.CMA.58(2009)}.

There are many improvements, refinements and generalizations of (\ref{M-C}).
For the first inequality in (\ref{M-C}), a nice refinement has appeared in 
\cite[3.4.6]{Mitrinovic.AI.1970}, which states that for $x\in \left( 0,\pi
/2\right) $, the inequalities%
\begin{equation*}
\left( \cos x\right) ^{1/3}<\cos \frac{x}{\sqrt{3}}<\frac{\sin x}{x}
\end{equation*}%
hold. Neuman presented an interesting chain of inequalities in \cite[Theorem
1]{Neuman-AIA-12} (also see \cite{Neuman.13(4)(2010)}, \cite{Lv.25(2012)}, 
\cite{Neuman.JMI-5(4)-2012}, \cite[(3.23)]{Chen.JMI.8.1.2014}), that is, the
inequalities%
\begin{eqnarray}
\left( \cos x\right) ^{1/3} &<&\left( \cos x\frac{\sin x}{x}\right)
^{1/4}<\left( \frac{\sin x}{\func{arctanh}\sin x}\right) ^{1/2}<\left( \frac{%
\cos x+\left( \sin x\right) /x}{2}\right) ^{1/2}  \notag \\
&<&\left( \frac{1+2\cos x}{3}\right) ^{1/2}<\left( \frac{1+\cos x}{2}\right)
^{2/3}<\frac{\sin x}{x}  \label{Neuman1}
\end{eqnarray}%
are valid for $x\in \left( 0,\pi /2\right) $.

For the second one in (\ref{M-C}), Yang \cite{Yang.GJM.2.1.2014} and Kl\'{e}%
n et al. \cite[Theorem 2.4]{Klen.JIA.2010} showed that for\emph{\ }$x\in
(0,\pi )$%
\begin{equation}
\frac{\sin x}{x}\leq \cos ^{3}\frac{x}{3}\leq \frac{2+\cos x}{3}.
\label{Klen-Yang}
\end{equation}%
Further, Yang \cite{Yang.JIA.2013.541} has shown that for $x\in \left( 0,\pi
/2\right) $ the inequalities%
\begin{equation}
\frac{\sin x}{x}<\left( \tfrac{2}{3}\cos \tfrac{x}{2}+\tfrac{1}{3}\right)
^{2}<\cos ^{3}\frac{x}{3}<\frac{2+\cos x}{3}  \label{Yang1}
\end{equation}%
are true.

By constructing a monotonic function $p\mapsto \left( \cos px\right)
^{1/\left( 3p^{2}\right) }$ ($p\in (0,1]$), Yang \cite{Yang.JMI.2013} showed
that the inequalities 
\begin{eqnarray}
\left( \cos x\right) ^{1/3} &<&\cos \frac{x}{\sqrt{3}}<\left( \cos \frac{x}{2%
}\right) ^{4/3}<\frac{\sin x}{x}  \label{Yang2} \\
&<&\left( \cos \frac{x}{3}\right) ^{3}<\left( \cos \frac{x}{4}\right)
^{16/3}<e^{-x^{2}/6}<\frac{2+\cos x}{3}  \notag
\end{eqnarray}%
are valid for $x\in \left( 0,\pi /2\right) $.

It is worth mentioning that Zhu \cite{Zhu.CMA.58(2009)} established a more
general result containing Cusa-type inequalities, which is recorded as
follows.

\noindent \textbf{Theorem Zhu (}\cite{Zhu.CMA.58(2009)}\textbf{)}\label{Zhu}%
\emph{Let }$0<x<\pi /2$\emph{. Then}

\emph{(i) if }$p\geq 1$\emph{, the double inequality}%
\begin{equation}
1-\xi +\xi \left( \cos x\right) ^{p}<\left( \frac{\sin x}{x}\right)
^{p}<1-\eta +\eta \left( \cos x\right) ^{p}  \label{Zhu1}
\end{equation}%
\emph{holds if and only if }$\eta \leq 1/3$\emph{\ and }$\xi \geq 1-\left(
2/\pi \right) ^{p}$\emph{;}

\emph{(ii) if }$0\leq p\leq 4/5$\emph{, the double inequality}%
\begin{equation}
1-\eta +\eta \left( \cos x\right) ^{p}<\left( \frac{\sin x}{x}\right)
^{p}<1-\xi +\xi \left( \cos x\right) ^{p}  \label{Zhu2}
\end{equation}%
\emph{holds if and only if }$\eta \geq 1/3$\emph{\ and }$\xi \leq 1-\left(
2/\pi \right) ^{p}$\emph{'}

\emph{(iii) if }$p<0$\emph{, the inequality}%
\begin{equation*}
\left( \frac{\sin x}{x}\right) ^{p}<1-\eta +\eta \left( \cos x\right) ^{p}
\end{equation*}%
\emph{holds if and only if }$\eta \geq 1/3$\emph{.}

As a consequence of Theorem Zhu, the double inequality%
\begin{equation*}
\left( \tfrac{2}{3}+\tfrac{1}{3}\left( \cos x\right) ^{4/5}\right) ^{5/4}<%
\frac{\sin x}{x}<\tfrac{2}{3}+\tfrac{1}{3}\cos x
\end{equation*}%
holds $0<x<\pi /2$, where $4/5$ is the best. He \cite{He.JSU.26.2.2011}\ and
Yang \cite{Yang.MIA.17.2.2014} showed independently that the double
inequality%
\begin{equation}
\left( \tfrac{2}{3}+\tfrac{1}{3}\left( \cos x\right) ^{4/5}\right) ^{5/4}<%
\frac{\sin x}{x}<\left( \tfrac{2}{3}+\tfrac{1}{3}\left( \cos x\right)
^{p_{0}^{\ast }}\right) ^{1/p_{0}^{\ast }},  \label{Yang3}
\end{equation}%
holds for $x\in \left( 0,\pi /2\right) $ with the best constants $4/5$ and $%
p_{0}^{\ast }=\log _{\pi /2}\left( 3/2\right) \approx 0.8979$.

The aim of this paper is to determine the best or better $p,q$ such that the
inequality%
\begin{equation}
\left( \frac{\sin x}{x}\right) ^{p}<\left( >\right) 1-\beta +\beta \cos ^{q}x
\label{m1}
\end{equation}%
or%
\begin{equation}
\frac{\sin x}{x}<\left( >\right) \left( 1-\beta +\beta \cos ^{q}x\right)
^{1/p}  \label{m2}
\end{equation}%
holds for $x\in \left( 0,\pi /2\right) $.

The paper is organized as follows. In Section 2, we investigate the
monotonicity of the function $T_{p,q}$ defined on $\left( 0,\pi /2\right) $
by%
\begin{equation}
T_{p,q}\left( x\right) =\frac{U_{p}\left( \frac{\sin x}{x}\right) }{%
U_{q}\left( \cos x\right) },  \label{T_p,q}
\end{equation}%
where $p,q\in \mathbb{R}$ and $U_{p}$ is defined on $\left( 0,1\right) $ by%
\begin{equation}
U_{p}\left( t\right) =\frac{1-t^{p}}{p}\text{ if }p\neq 0\text{ and }%
U_{0}\left( t\right) =-\ln t.  \label{U_p}
\end{equation}%
In Section 3, by using the monotonicity of $T_{p,q}$ on $\left( 0,\pi
/2\right) $, we prove some sharp Cusa type inequalities for trigonometric
functions for certain $p,q$. It is not only to generalize Zhu's results, but
also present many new and interesting inequalities for trigonometric
functions. In the last section, as applications, some new inequalities for
arc sine function and bivariate means are presented.

\section{Monotonicity}

We begin with the following simple assertion.

\begin{lemma}
\label{Lemma u_p}Let the function $U_{p}$ defined on $\left( 0,1\right) $ by
(\ref{U_p}). Then $p\mapsto U_{p}\left( t\right) $ is decreasing on $\mathbb{%
R}$ and $U_{p}\left( t\right) >0$ for $t\in \left( 0,1\right) $.
\end{lemma}

\begin{proof}
For $p\neq 0$, differentiation yields%
\begin{equation*}
\frac{\partial }{\partial p}U_{p}\left( t\right) =\frac{1}{p^{2}}\left(
t^{p}-1\right) -\frac{1}{p}t^{p}\ln t=\frac{t^{p}}{p^{2}}\left( \ln
t^{-p}-\left( t^{-p}-1\right) \right) <0,
\end{equation*}%
where the last inequality holds due to $\ln x\leq \left( x-1\right) $ for $%
x>0$.

Employing the decreasing property, we get $U_{p}\left( t\right)
>\lim_{p\rightarrow \infty }U_{p}\left( t\right) =0$, which proves the lemma.
\end{proof}

For $x\in (0,\pi /2)$, we denote by%
\begin{equation*}
S_{p}\left( x\right) :=U_{p}\left( \tfrac{\sin x}{x}\right) \text{ \ and \ }%
C_{p}\left( x\right) :=U_{p}\left( \cos x\right)
\end{equation*}%
due to $\left( \sin x\right) /x,\cos x\in \left( 0,1\right) $. Then we have%
\begin{eqnarray}
S_{p}\left( x\right) &=&\frac{1-\left( \frac{\sin x}{x}\right) ^{p}}{p}\text{
if }p\neq 0\text{ \ and \ }S_{0}\left( x,p\right) =-\ln \frac{\sin x}{x}%
\text{ if }p=0,  \label{Sp} \\
C_{p}\left( x\right) &=&\frac{1-\cos ^{p}x}{p}\text{ if }p\neq 0\text{ \ and
\ }C_{0}\left( x,p\right) =-\ln \left( \cos x\right) \text{ if }p=0.
\label{Cp}
\end{eqnarray}%
And then, the function $x\mapsto T_{p,q}\left( x\right) =U_{p}\left( \frac{%
\sin x}{x}\right) /U_{q}\left( \cos x\right) =S_{p}\left( x\right)
/C_{q}\left( x\right) $ can be expressed as%
\begin{equation}
T_{p,q}\left( x\right) =\left\{ 
\begin{array}{ll}
\frac{q}{p}\frac{1-\left( \frac{\sin x}{x}\right) ^{p}}{1-\cos ^{q}x} & 
\text{if }pq\neq 0, \\ 
\frac{1}{p}\frac{\left( \frac{\sin x}{x}\right) ^{p}-1}{\ln \left( \cos
x\right) } & \text{if }p\neq 0,q=0, \\ 
q\frac{\ln \frac{\sin x}{x}}{\cos ^{q}x-1} & \text{if }p=0,q\neq 0, \\ 
\frac{\ln \frac{\sin x}{x}}{\ln \left( \cos x\right) } & \text{if }p=q=0.%
\end{array}%
\right.  \label{T}
\end{equation}%
In order to investigate the monotonicity of the function $T_{p,q}$, we first
recall the following important lemmas.

\begin{lemma}[\protect\cite{Vamanamurthy.183.1994}, \protect\cite%
{Anderson.New York. 1997}]
\label{Lemma monotonicity of ratio}Let $f,g:\left[ a,b\right] \rightarrow 
\mathbb{R}$ be two continuous functions which are differentiable on $\left(
a,b\right) $. Further, let $g^{\prime }\neq 0$ on $\left( a,b\right) $. If $%
f^{\prime }/g^{\prime }$ is increasing (or decreasing) on $\left( a,b\right) 
$, then so are the functions 
\begin{equation*}
x\mapsto \frac{f\left( x\right) -f\left( a\right) }{g\left( x\right)
-g\left( a\right) }\text{ \ \ \ and \ \ \ }x\mapsto \frac{f\left( x\right)
-f\left( b\right) }{g\left( x\right) -g\left( b\right) }.
\end{equation*}
\end{lemma}

The following lemma is cruial to prove certain best inequalities, which is
inspired by part (iv) of proof of Theorem 6 in \cite%
{Yang.arXiv:1304.5369.2013}.

\begin{lemma}
\label{Lemma Yang}Suppose that $f,g:\left[ a,b\right] \rightarrow \mathbb{R}$
are two continuous functions which are differentiable on $\left( a,b\right) $
and $g^{\prime }\neq 0$ on $\left( a,b\right) $. If $f^{\prime }/g^{\prime }$
is increasing (decreasing) on $\left( a,x_{0}\right) $ and decreasing
(increasing) on $\left( x_{0},b\right) $, and%
\begin{equation}
\frac{f\left( b\right) -f\left( a\right) }{g\left( b\right) -g\left(
a\right) }\geq \left( \leq \right) \frac{f^{\prime }\left( a^{+}\right) }{%
g^{\prime }\left( a^{+}\right) }=\lambda \neq \pm \infty ,  \label{lambda}
\end{equation}%
then the inequality%
\begin{equation}
\frac{f\left( x\right) -f\left( a\right) }{g\left( x\right) -g\left(
a\right) }>\left( <\right) \lambda  \label{>(<)lambda}
\end{equation}%
holds for all $x\in \left( a,b\right) $.
\end{lemma}

\begin{proof}
Without loss of generality, we assume that $g^{\prime }>0$ on $\left(
a,b\right) $.

For $x\in \left( a,x_{0}\right) $, by Lemma \ref{Lemma monotonicity of ratio}%
, since $f^{\prime }/g^{\prime }$ is increasing (decreasing) on $\left(
a,x_{0}\right) $, so is the function%
\begin{equation*}
x\mapsto \frac{f\left( x\right) -f\left( a\right) }{g\left( x\right)
-g\left( a\right) }.
\end{equation*}%
Then we get that for $x\in (a,x_{0}]$%
\begin{equation}
\frac{f\left( x\right) -f\left( a\right) }{g\left( x\right) -g\left(
a\right) }>\left( <\right) \lim_{x\rightarrow a^{+}}\frac{f\left( x\right)
-f\left( a\right) }{g\left( x\right) -g\left( a\right) }=\frac{f^{\prime
}\left( a^{+}\right) }{g^{\prime }\left( a^{+}\right) }=\lambda .
\label{(a,x0]}
\end{equation}%
That is, the inequality (\ref{>(<)lambda}) holds for $x\in (a,x_{0}]$.

On the other hand, from Lemma \ref{Lemma monotonicity of ratio}, that $%
f^{\prime }/g^{\prime }$ is decreasing (increasing) on $\left(
x_{0},b\right) $ means that so is the function%
\begin{equation*}
x\mapsto \frac{f\left( x\right) -f\left( b\right) }{g\left( x\right)
-g\left( b\right) },
\end{equation*}%
and hence we have%
\begin{equation*}
\frac{f\left( x\right) -f\left( b\right) }{g\left( x\right) -g\left(
b\right) }<\left( >\right) \frac{f\left( x_{0}\right) -f\left( b\right) }{%
g\left( x_{0}\right) -g\left( b\right) }\text{ for }x\in \left(
x_{0},b\right) ,
\end{equation*}%
which can be rewritten as%
\begin{equation*}
f\left( x\right) >\left( <\right) f\left( b\right) +\frac{f\left(
x_{0}\right) -f\left( b\right) }{g\left( x_{0}\right) -g\left( b\right) }%
\left( g\left( x\right) -g\left( b\right) \right) :=\phi \left( x\right) ,
\end{equation*}%
due to assumpation that $g^{\prime }>0$ on $\left( a,b\right) $. Clearly, in
order to prove the desired inequality, it suffices to prove%
\begin{equation*}
\phi \left( x\right) >\left( <\right) f\left( a\right) +\lambda \left(
g\left( x\right) -g\left( a\right) \right) \text{ for }x\in \left(
x_{0},b\right) .
\end{equation*}

Since $x_{0}$ statifies the relation (\ref{(a,x0]}), that is,%
\begin{equation*}
f\left( x_{0}\right) >\left( <\right) f\left( a\right) +\lambda \left(
g\left( x_{0}\right) -g\left( a\right) \right) ,
\end{equation*}%
which together with (\ref{lambda}), that is,%
\begin{equation*}
f\left( b\right) \geq \left( \leq \right) f\left( a\right) +\lambda \left(
g\left( b\right) -g\left( a\right) \right) ,
\end{equation*}%
leads to%
\begin{eqnarray*}
\phi \left( x\right) &=&\frac{g\left( x\right) -g\left( b\right) }{g\left(
x_{0}\right) -g\left( b\right) }f\left( x_{0}\right) +\frac{g\left(
x_{0}\right) -g\left( x\right) }{g\left( x_{0}\right) -g\left( b\right) }%
f\left( b\right) \\
&>&\left( <\right) \frac{g\left( x\right) -g\left( b\right) }{g\left(
x_{0}\right) -g\left( b\right) }\left( f\left( a\right) +\lambda \left(
g\left( x_{0}\right) -g\left( a\right) \right) \right) \\
&&+\frac{g\left( x_{0}\right) -g\left( x\right) }{g\left( x_{0}\right)
-g\left( b\right) }\left( f\left( a\right) +\lambda \left( g\left( b\right)
-g\left( a\right) \right) \right) \\
&=&f\left( a\right) +\lambda \left( g\left( x\right) -g\left( a\right)
\right) .
\end{eqnarray*}%
This means that the inequality (\ref{>(<)lambda}) also holds for $x\in
\left( x_{0},b\right) $. Thus the proof is completed.
\end{proof}

\begin{lemma}[\protect\cite{Biernacki.9.1955}]
\label{Lemma monotocity of A/B}Let $a_{n}$ and $b_{n}$ $(n=0,1,2,...)$ be
real numbers and let the power series $A\left( t\right) =\sum_{n=0}^{\infty
}a_{n}t^{n}$ and $B\left( t\right) =\sum_{n=0}^{\infty }b_{n}t^{n}$ be
convergent for $|t|<R$. If $b_{n}>0$ for $n=0,1,2,...$, and $a_{n}/b_{n}$ is
strictly increasing (or decreasing) for $n=0,1,2,...$, then the function $%
A\left( t\right) /B\left( t\right) $ is strictly increasing (or decreasing)
on $\left( 0,R\right) $.
\end{lemma}

\begin{lemma}[{\protect\cite[pp.227-229]{Handbook.math.1979}}]
\label{Lemma expansion t}We have%
\begin{eqnarray}
\frac{1}{\sin x} &=&\frac{1}{x}+\sum_{n=1}^{\infty }\frac{2^{2n}-2}{\left(
2n\right) !}|B_{2n}|x^{2n-1}\text{,\ }|x|<\pi  \label{1/sinx} \\
\cot x &=&\frac{1}{x}-\sum_{n=1}^{\infty }\frac{2^{2n}}{\left( 2n\right) !}%
|B_{2n}|x^{2n-1}\text{, \ }|x|<\pi ,  \label{cotx} \\
\frac{1}{\sin ^{2}x} &=&\frac{1}{x^{2}}+\sum_{n=1}^{\infty }\frac{\left(
2n-1\right) 2^{2n}}{\left( 2n\right) !}|B_{2n}|x^{2n-2}\text{, \ }|x|<\pi ,
\label{1/sin^2x} \\
\tan x &=&\sum_{n=1}^{\infty }\frac{2^{2n}-1}{\left( 2n\right) !}%
2^{2n}|B_{2n}|x^{2n-1}\text{, \ }|x|<\pi /2.  \label{tanx}
\end{eqnarray}%
where $B_{n}$ is the Bernoulli number.
\end{lemma}

Now we are in position to prove the monotonicity of $T_{p,q}$. Clearly, $%
T_{p,q}\left( x\right) $ can be written as%
\begin{equation*}
T_{p,q}\left( x\right) =\frac{S_{p}\left( x\right) }{C_{q}\left( x\right) }=%
\frac{S_{p}\left( x\right) -S_{p}\left( 0^{+}\right) }{C_{q}\left( x\right)
-C_{q}\left( 0^{+}\right) }.
\end{equation*}%
For $pq\neq 0$, differentiation yields%
\begin{equation}
\frac{S_{p}^{\prime }\left( x\right) }{C_{q}^{\prime }\left( x\right) }=%
\frac{\cos ^{1-q}x}{x^{2}\sin x}\left( \frac{\sin x}{x}\right) ^{p-1}\left(
\sin x-x\cos x\right) :=f_{1}(x)  \label{f1}
\end{equation}%
\begin{equation}
f_{1}^{\prime }(x)=-\frac{1}{x^{2}\sin ^{3}x\cos ^{q}x}\left( \frac{\sin x}{x%
}\right) ^{p}\times f_{2}\left( x\right) ,  \label{df1}
\end{equation}%
where%
\begin{equation}
f_{2}\left( x\right) =pA\left( x\right) -qB\left( x\right) +C\left( x\right)
,  \label{f2}
\end{equation}%
in which 
\begin{subequations}
\begin{eqnarray}
A\left( x\right) &=&\left( \sin x-x\cos x\right) ^{2}\cos x>0,  \label{A} \\
B\left( x\right) &=&x\left( \sin x-x\cos x\right) \sin ^{2}x>0  \label{B} \\
C\left( x\right) &=&-\left( 2x^{2}\cos x-x\sin x-\cos x\sin ^{2}x\right) >0,
\label{C}
\end{eqnarray}%
here $C\left( x\right) >0$ due to 
\end{subequations}
\begin{equation*}
C\left( x\right) =x^{2}\left( \cos x\right) \left( \frac{\sin ^{2}x}{x^{2}}+%
\frac{\tan x}{x}-2\right) >0
\end{equation*}%
by Wilker inequality (see \cite{Wilker.96(1)(1989)}). It is easy to verify
that (\ref{f1}), (\ref{df1}) and (\ref{f2}) are true for $pq=0$.

We see clearly that, by Lemma \ref{Lemma monotonicity of ratio}, if we can
prove $f_{2}\left( x\right) \leq (\geq )0$ for all $x\in (0,\pi /2)$ then $%
T_{p,q}$ defined by (\ref{T}) is increasing (decreasing) on $(0,\pi /2)$. In
order to prove it, we need to the expansions of $A\left( x\right) ,B\left(
x\right) $ and $C\left( x\right) $. Using Lemma \ref{Lemma expansion t} we
get%
\begin{eqnarray}
\frac{A(x)}{\sin ^{2}x\cos x} &=&\frac{x^{2}\cos ^{3}x-2x\cos ^{2}x\sin
x+\cos x\sin ^{2}x}{\sin ^{2}x\cos x}=x^{2}\frac{1}{\sin ^{2}x}-x^{2}-2x%
\frac{\cos x}{\sin x}+1  \notag \\
&=&x^{2}\left( \frac{1}{x^{2}}+\sum_{n=1}^{\infty }\frac{\left( 2n-1\right)
2^{2n}}{\left( 2n\right) !}|B_{2n}|x^{2n-2}\right) -x^{2}  \notag \\
&&-2x\left( \frac{1}{x}-\sum_{n=1}^{\infty }\frac{2^{2n}}{\left( 2n\right) !}%
|B_{2n}|x^{2n-1}\right) +1  \notag \\
&=&\sum_{n=2}^{\infty }\frac{\left( 2n+1\right) 2^{2n}}{\left( 2n\right) !}%
|B_{2n}|x^{2n}:=\sum_{n=2}^{\infty }a_{n}x^{2n},  \label{Ae}
\end{eqnarray}%
\begin{eqnarray}
\frac{B\left( x\right) }{\sin ^{2}x\cos x} &=&\frac{x\sin ^{3}x-x^{2}\cos
x\sin ^{2}x}{\sin ^{2}x\cos x}=x\frac{\sin x}{\cos x}-x^{2}  \notag \\
&=&\sum_{n=1}^{\infty }\frac{2^{2n}-1}{\left( 2n\right) !}%
2^{2n}|B_{2n}|x^{2n}-x^{2}  \notag \\
&=&\sum_{n=2}^{\infty }\frac{\left( 2^{2n}-1\right) 2^{2n}}{\left( 2n\right)
!}|B_{2n}|x^{2n}:=\sum_{n=2}^{\infty }b_{n}x^{2n},  \label{Be}
\end{eqnarray}%
\begin{eqnarray}
\frac{C\left( x\right) }{\sin ^{2}x\cos x} &=&\frac{-2x^{2}\cos x+x\sin
x+\cos x\sin ^{2}x}{\sin ^{2}x\cos x}=-2x^{2}\frac{1}{\sin ^{2}x}+2x\frac{1}{%
\sin 2x}+1  \notag \\
&=&-2x^{2}\left( \frac{1}{x^{2}}+\sum_{n=1}^{\infty }\frac{\left(
2n-1\right) 2^{2n}}{\left( 2n\right) !}|B_{2n}|x^{2n-2}\right)  \notag \\
&&+2x\left( \frac{1}{2x}+\sum_{n=1}^{\infty }\frac{2^{2n}-2}{\left(
2n\right) !}|B_{2n}|2^{2n-1}x^{2n-1}\right) +1  \notag \\
&=&\sum_{n=2}^{\infty }\frac{\left( 2^{2n}-4n\right) 2^{2n}}{\left(
2n\right) !}|B_{2n}|x^{2n}:=\sum_{n=2}^{\infty }c_{n}x^{2n}.  \label{Ce}
\end{eqnarray}

\begin{lemma}
\label{Lemma g1}Let $g_{1}$ be defined on $(0,\pi /2)$ by%
\begin{equation}
g_{1}\left( x\right) =\frac{qB\left( x\right) -C\left( x\right) }{A\left(
x\right) }.  \label{g1}
\end{equation}%
where $A\left( x\right) ,B\left( x\right) $ and $C\left( x\right) $ are
defined by (\ref{A}), (\ref{B}) and (\ref{C}), respectively. Then

(i) $g_{1}$ is increasing on $(0,\pi /2)$ if $q\geq 1$, and we have%
\begin{equation}
3q-\frac{8}{5}<g_{1}\left( x\right) <\left\{ 
\begin{array}{ll}
\infty & \text{if }q>1, \\ 
\frac{\pi ^{2}}{4}-1 & \text{if }q=1;%
\end{array}%
\right.  \label{g1boundq>=}
\end{equation}

(ii) $g_{1}$ is decreasing on $(0,\pi /2)$ if $q\leq 34/35$, and we have%
\begin{equation}
-\infty <g_{1}\left( x\right) <3q-\frac{8}{5}.  \label{g1boundq<}
\end{equation}
\end{lemma}

\begin{proof}
Using (\ref{Ae}), (\ref{Be}) and (\ref{Ce}) we have%
\begin{equation*}
g_{1}\left( x\right) =\tfrac{qB\left( x\right) -C\left( x\right) }{A\left(
x\right) }=\tfrac{\sum_{n=2}^{\infty }\left( qb_{n}-c_{n}\right) x^{2n}}{%
\sum_{n=2}^{\infty }a_{n}x^{2n}}\text{ \ and \ }\frac{qb_{n}-c_{n}}{a_{n}}=%
\tfrac{\left( 2^{2n}-1\right) q-\left( 2^{2n}-4n\right) }{\left( 2n+1\right) 
}.
\end{equation*}%
And then,%
\begin{eqnarray*}
\frac{qb_{n+1}-c_{n+1}}{a_{n+1}}-\frac{qb_{n}-c_{n}}{a_{n}} &=&\tfrac{\left(
2^{2\left( n+1\right) }-1\right) q-\left( 2^{2\left( n+1\right) }-4\left(
n+1\right) \right) }{\left( 2\left( n+1\right) +1\right) }-\tfrac{\left(
2^{2n}-1\right) q-\left( 2^{2n}-4n\right) }{\left( 2n+1\right) } \\
&=&\frac{\left( 6n+1\right) 4^{n}+2}{\left( 2n+3\right) \left( 2n+1\right) }%
\left( q-1+\frac{6}{\left( 6n+1\right) 4^{n}+2}\right) .
\end{eqnarray*}%
From this it is obtained that%
\begin{equation*}
\frac{qb_{n+1}-c_{n+1}}{a_{n+1}}-\frac{qb_{n}-c_{n}}{a_{n}}\left\{ 
\begin{array}{cc}
>0 & \text{if }q\geq \sup_{n\in \mathbb{N},n\geq 2}\left( 1-\frac{6}{\left(
6n+1\right) 4^{n}+2}\right) =1, \\ 
<0 & \text{if }q\leq \inf_{n\in \mathbb{N},n\geq 2}\left( 1-\frac{6}{\left(
6n+1\right) 4^{n}+2}\right) =\frac{34}{35}.%
\end{array}%
\right.
\end{equation*}

In the case of $q\geq 1$, we see that $\left( qb_{n}-c_{n}\right) /a_{n}$ is
increasing with $n\geq 2$, and by Lemma \ref{Lemma monotocity of A/B} it is
seen that $g_{1}$ is increasing on $(0,\pi /2)$. Hence, we have%
\begin{equation*}
3q-\frac{8}{5}=\lim_{x\rightarrow 0^{+}}g_{1}\left( x\right) <g_{1}\left(
x\right) <\lim_{x\rightarrow \pi /2^{-}}g_{1}\left( x\right) =\left\{ 
\begin{array}{ll}
\infty & \text{if }q>1, \\ 
\frac{\pi ^{2}}{4}-1 & \text{if }q=1.%
\end{array}%
\right.
\end{equation*}

When $q\leq 34/35$, the sequence $\left( qb_{n}-c_{n}\right) /a_{n}$ is
decreasing with $n\geq 2$, and so is the function $\left( qB-C\right) /A$ on 
$(0,\pi /2)$. Hence, we have (\ref{g1boundq<}).

This lemma is proved.
\end{proof}

\begin{lemma}
\label{Lemma g2}Let $g_{2}$ be defined on $(0,\pi /2)$ by%
\begin{equation}
g_{2}\left( x\right) =\frac{C\left( x\right) -\frac{8}{5}A\left( x\right) }{%
B\left( x\right) -3A\left( x\right) }.  \label{g2}
\end{equation}%
where $A\left( x\right) ,B\left( x\right) $ and $C\left( x\right) $ are
defined by (\ref{A}), (\ref{B}) and (\ref{C}), respectively. Then (i) $%
B\left( x\right) -3A\left( x\right) >0$ for $x\in (0,\pi /2)$; (ii) $g_{2}$
is increasing on $(0,\pi /2)$, and we have $34/35<g_{2}\left( x\right) <1$.
\end{lemma}

\begin{proof}
By using (\ref{Ae}), (\ref{Be}) and (\ref{Ce}), we get%
\begin{equation*}
g_{2}\left( x\right) =\frac{C\left( x\right) -\frac{8}{5}A\left( x\right) }{%
B\left( x\right) -3A\left( x\right) }=\frac{\sum_{n=3}^{\infty }\left( c_{n}-%
\frac{8}{5}a_{n}\right) x^{2n}}{\sum_{n=3}^{\infty }\left(
b_{n}-3a_{n}\right) x^{2n}}\text{ \ and \ }\frac{c_{n}-\frac{8}{5}a_{n}}{%
b_{n}-3a_{n}}=\frac{2^{2n}-\frac{36}{5}n-\frac{8}{5}}{2^{2n}-6n-4}.
\end{equation*}%
(i) In order for $B\left( x\right) -3A\left( x\right) >0$ to be true for $%
x\in (0,\pi /2)$, it suffices that $b_{n}-3a_{n}>0$ for $n\geq 3$. Employing
binomial expansion yields%
\begin{equation*}
b_{n}-3a_{n}=2^{2n}-6n-4>1+2n+\frac{2n\left( 2n-1\right) }{2}-6n-4=\left(
2n+1\right) \left( n-3\right) \geq 0.
\end{equation*}%
(ii) By Lemma \ref{Lemma monotocity of A/B}, to prove $g_{2}$ is increasing
on $(0,\pi /2)$, it is enough to show that for $n\geq 3$ 
\begin{equation*}
\frac{c_{n+1}-\frac{8}{5}a_{n+1}}{b_{n+1}-3a_{n+1}}-\frac{c_{n}-\frac{8}{5}%
a_{n}}{b_{n}-3a_{n}}>0.
\end{equation*}%
A direct computation leads to 
\begin{eqnarray*}
\frac{c_{n+1}-\frac{8}{5}a_{n+1}}{b_{n+1}-3a_{n+1}}-\frac{c_{n}-\frac{8}{5}%
a_{n}}{b_{n}-3a_{n}} &=&\frac{2^{2n+2}-\frac{36}{5}n-\frac{44}{5}}{%
2^{2n+2}-6n-10}-\frac{2^{2n}-\frac{36}{5}n-\frac{8}{5}}{2^{2n}-6n-4} \\
&=&\frac{6}{5}\frac{\left( 3n-7\right) 2^{2n}+16}{\left(
2^{2n+2}-6n-10\right) \left( 2^{2n}-6n-4\right) }>0,
\end{eqnarray*}%
which shows that the sequence $\left( c_{n}-8a_{n}/5\right) /\left(
b_{n}-3a_{n}\right) $ is increasing with $n\geq 3$, and by Lemma \ref{Lemma
monotocity of A/B} it is seen that $g_{2}$ is increasing on $(0,\pi /2)$.
Consequently, we get%
\begin{equation*}
\frac{34}{35}=\lim_{x\rightarrow 0^{+}}g_{2}\left( x\right) <g_{2}\left(
x\right) <\lim_{x\rightarrow \pi /2^{-}}g_{2}\left( x\right) =1,
\end{equation*}%
which proves the lemma.
\end{proof}

Now we state and prove the monotonicity of $T_{p,q}$.

\begin{proposition}
\label{P main1}Let $T_{p,q}$ be defined on $(0,\pi /2$ by (\ref{T}). Then

(i) when $q>1$, $T_{p,q}$ is increasing on $(0,\pi /2)$ for $p\leq 3q-8/5$;

(ii) when $q=1$, $T_{p,q}$ is increasing on $(0,\pi /2)$ for $p\leq 7/5$ and
decreasing on $(0,\pi /2)$ for $p\geq \pi ^{2}/4-1$;

(iii) when $34/35<q<1$, $T_{p,q}$ is decreasing on $(0,\pi /2)$ for $p\geq
\pi ^{2}/4-1$;

(iv) when $q\leq 34/35$, $T_{p,q}$ is decreasing on $(0,\pi /2)$ for $p\geq
3q-8/5$.
\end{proposition}

\begin{proof}
As mentioned previously, to derive the monotonicity of $T_{p,q}$, it
suffices to deal with the sings of $f_{2}\left( x\right) $ on $\left( 0,\pi
/2\right) $. To this end, we need to write $f_{2}\left( x\right) $ in the
form of%
\begin{equation}
f_{2}\left( x\right) =pA\left( x\right) -qB\left( x\right) +C\left( x\right)
=A\left( x\right) \left( p-g_{1}\left( x\right) \right) ,  \label{f2another}
\end{equation}%
where $g_{1}\left( x\right) $ is defined by (\ref{g1}). Then, $\func{sgn}%
f_{2}\left( x\right) =\func{sgn}\left( p-g_{1}\left( x\right) \right) $ due
to $A\left( x\right) >0$ for $x\in (0,\pi /2)$.

(i) When $q>1$, it is obtained from Lemma \ref{Lemma g1} that $f_{2}\left(
x\right) <0$ for $x\in \left( 0,\pi /2\right) $ due to 
\begin{equation*}
p-g_{1}\left( x\right) \leq p-\left( 3q-\frac{8}{5}\right) \leq 0
\end{equation*}%
Utilizing the relation (\ref{df1})\ and Lemma \ref{Lemma monotonicity of
ratio} we get the first assertion in this theorem.

(ii) When $q=1$, similarly, it is acquired that $f_{2}\left( x\right) <0$
due to $\left( p-g_{1}\left( x\right) \right) \leq 3\times 1-8/5=7/5$ and $%
f_{2}\left( x\right) >0$ due to 
\begin{equation*}
p-g_{1}\left( x\right) \geq p-\left( \frac{\pi ^{2}}{4}-1\right) \geq 0\text{%
.}
\end{equation*}%
Make use of the relation (\ref{df1})\ and Lemma \ref{Lemma monotonicity of
ratio} again, the second assertion in this theorem follows.

(iii) When $34/35<q<1$, we see that $f_{2}\left( x\right) >0$ in view of%
\begin{equation*}
p-g_{1}\left( x\right) =p-\tfrac{qB\left( x\right) -C\left( x\right) }{%
A\left( x\right) }>p-\tfrac{1\times B\left( x\right) -C\left( x\right) }{%
A\left( x\right) }\geq p-\left( \frac{\pi ^{2}}{4}-1\right) \geq 0\text{.}
\end{equation*}

(iv) When $q\leq 34/35$, it can be proved in the same way.

Thus we complete the proof.
\end{proof}

Letting $p=3q-8/5$ in Proposition \ref{P main1}, we have

\begin{corollary}
\label{Corollary p=3q-8/5}Let $T_{p,q}$ be defined on $(0,\pi /2$ by (\ref{T}%
). Then $T_{3q-8/5,q}$ is increasing on $(0,\pi /2)$ for $q\geq 1$ and
decreasing for $q\leq 34/35$.
\end{corollary}

Since $p\leq \left( \geq \right) 3q-8/5$ is equivalent to $q\geq \left( \leq
\right) p/3+8/15$, Proposition \ref{P main1} can be restated as follows.

\begin{proposition}
\label{P main2}Let $T_{p,q}$ be defined on $(0,\pi /2$ by (\ref{T}). Then

(i) $T_{p,q}$ is increasing on $(0,\pi /2)$ if $q\geq \max \left(
1,p/3+8/15\right) $;

(ii) $T_{p,q}$ is decreasing on $(0,\pi /2)$ if $q\leq \min \left(
34/35,p/3+8/15\right) $ or $p\geq \pi ^{2}/4-1$ and $34/35<q\leq 1$.
\end{proposition}

Due to%
\begin{eqnarray*}
\max \left( 1,\tfrac{p}{3}+\tfrac{8}{15}\right) &=&\left\{ 
\begin{array}{cc}
\frac{p}{3}+\frac{8}{15} & \text{if }p\geq \frac{7}{5}, \\ 
1 & \text{if }p<\frac{7}{5},%
\end{array}%
\right. \\
\min \left( \tfrac{34}{35},\tfrac{p}{3}+\tfrac{8}{15}\right) &=&\left\{ 
\begin{array}{cc}
\tfrac{34}{35} & \text{if }p\geq \frac{46}{35}, \\ 
\tfrac{p}{3}+\tfrac{8}{15} & \text{if }p<\frac{46}{35},%
\end{array}%
\right.
\end{eqnarray*}%
Position \ref{P main2} also can be restated in another equivalent assertion.

\begin{proposition}
\label{P main3}Let $T_{p,q}$ be defined on $(0,\pi /2$ by (\ref{T}). Then

(i) if $p\geq \pi ^{2}/4-1$, then $T_{p,q}$ is increasing on $(0,\pi /2)$
for $q\geq p/3+8/15$ and decreasing on $(0,\pi /2)$ for $q\leq 1$;

(ii) if $p\in \lbrack 7/5,\pi ^{2}/4-1)$, then $T_{p,q}$ is increasing on $%
(0,\pi /2)$ for $q\geq p/3+8/15$ and decreasing on $(0,\pi /2)$ for $q\leq
34/35$;

(iii) if $p\in \lbrack 46/35,7/5)$, then $T_{p,q}$ is increasing on $(0,\pi
/2)$ for $q\geq 1$ and decreasing on $(0,\pi /2)$ for $q\leq 34/35$;

(iv) if $p<46/35$, then $T_{p,q}$ is increasing on $(0,\pi /2)$ for $q\geq 1$
and decreasing on $(0,\pi /2)$ for $q\leq p/3+8/15$.
\end{proposition}

Let $p=kq$. Then solving the simultaneous inequalities $q\geq \max \left(
1,kq/3+8/15\right) $ and $q\leq \min \left( 34/35,kq/3+8/15\right) $ for $q$
give 
\begin{equation*}
\left\{ 
\begin{array}{ll}
q\geq \frac{8}{5\left( 3-k\right) } & \text{if }k\in \lbrack \frac{7}{5},3),
\\ 
q\geq 1 & \text{if }k\in (-\infty ,\frac{7}{5}]%
\end{array}%
\right. \text{ \ and \ }\left\{ 
\begin{array}{ll}
\frac{8}{5\left( 3-k\right) }\leq q\leq 34/35 & \text{if }k\in \left(
3,\infty \right) , \\ 
q\leq \frac{34}{35} & \text{if }k\in \lbrack \frac{23}{17},3), \\ 
q\leq \frac{8}{5\left( 3-k\right) } & \text{if }k\in (-\infty ,\frac{23}{17}%
],%
\end{array}%
\right.
\end{equation*}%
respectively; while the solution of the simultaneous inequalities $kq\geq
\pi ^{2}/4-1$ and $34/35<q\leq 1$ is:%
\begin{equation*}
\left\{ 
\begin{array}{ll}
34/35<q\leq 1 & \text{if }k\geq \frac{35\pi ^{2}-140}{136}\approx 1.5106, \\ 
\frac{\pi ^{2}-4}{4k}\leq q\leq 1 & \text{if }k\in \left( \frac{\pi ^{2}-4}{4%
},\frac{35\pi ^{2}-140}{136}\right) .%
\end{array}%
\right.
\end{equation*}%
By Proposition \ref{P main2}, we have

\begin{corollary}
\label{Corollary p=kq}Let $T_{p,q}$ be defined on $(0,\pi /2)$ by (\ref{T}).
Then

(i) when $k\in \left( 3,\infty \right) $, $T_{kq,q}$ is decreasing for $%
8/\left( 5\left( 3-k\right) \right) \leq q\leq 34/35$;

(ii) when $k\in \lbrack \left( 35\pi ^{2}-140\right) /136,3)$, $T_{kq,q}$ is
increasing for $q\geq 8/\left( 5\left( 3-k\right) \right) $ and decreasing
for $q\leq 1$;

(iii) when $k\in \lbrack \pi ^{2}/4-1,\left( 35\pi ^{2}-140\right) /136)$, $%
T_{kq,q}$ is increasing for $q\geq 8/\left( 5\left( 3-k\right) \right) $ and
decreasing for $q\leq 34/35$ or $\left( \pi ^{2}/4-1\right) /k\leq q\leq 1$;

(iv) when $k\in \lbrack 7/5,\pi ^{2}/4-1),T_{kq,q}$ is increasing for $q\geq
8/\left( 5\left( 3-k\right) \right) $ and decreasing for $q\leq 34/35$;

(v) when $k\in \lbrack 23/17,7/5)$, $T_{kq,q}$ is increasing for $q\geq 1$
and decreasing for $q\leq 34/35$;

(vi) when $k\in \left( -\infty ,23/17\right) $, $T_{kq,q}$ is increasing for 
$q\geq 1$ and decreasing for $q\leq 8/\left( 5\left( 3-k\right) \right) $.
\end{corollary}

\section{Results and proofs}

In this section, we will give some new inequalities involving trigonometric
functions by using monotonicity theorems given in previous section. For
clarity of expressions, we will directly write $S_{p}\left( x\right)
,C_{q}\left( x\right) ,T_{p,q}\left( x\right) $ etc. by their general
formulas, and if $pq=0$, then we regard them as limits at $p=0$ or $q=0$,
unless otherwise specified.

\subsection{In the general case}

A simple computation yields 
\begin{equation*}
T_{p,q}\left( 0^{+}\right) =\frac{1}{3}\text{ \ and \ }T_{p,q}\left( \frac{%
\pi }{2}^{-}\right) =\left\{ 
\begin{array}{ll}
\frac{q}{p}\left( 1-\left( \frac{2}{\pi }\right) ^{p}\right) & \text{if }%
q>0,p\neq 0, \\ 
-q\ln \frac{2}{\pi } & \text{if }q>0,p=0, \\ 
0 & \text{if }q\leq 0.%
\end{array}%
\right.
\end{equation*}%
And then, by Proposition \ref{P main1}, we obtain the following theorem.

\begin{theorem}
\label{MT1 general}Let $x\in (0,\pi /2)$.

(i) If $q\geq 1$ and $p\leq 3q-8/5$, then the inequalities 
\begin{eqnarray}
\left( \tfrac{2}{\pi }\right) ^{p}+\left( 1-\left( \tfrac{2}{\pi }\right)
^{p}\right) \cos ^{q}x &<&\left( \tfrac{\sin x}{x}\right) ^{p}<1-\tfrac{p}{3q%
}+\tfrac{p}{3q}\cos ^{q}x\text{ if }p>0,  \label{M1p>0} \\
\left( \tfrac{2}{\pi }\right) ^{1-\cos ^{q}x} &<&\tfrac{\sin x}{x}<\exp 
\tfrac{\cos ^{q}x-1}{3q}\text{ if }p=0,  \label{M1p=0} \\
1-\tfrac{p}{3q}+\tfrac{p}{3q}\cos ^{q}x &<&\left( \tfrac{\sin x}{x}\right)
^{p}<\left( \tfrac{2}{\pi }\right) ^{p}+\left( 1-\left( \tfrac{2}{\pi }%
\right) ^{p}\right) \cos ^{q}x\text{ if }p<0  \label{M1p<0}
\end{eqnarray}%
hold, where $1/3$ and $q\left( 1-\left( 2/\pi \right) ^{p}\right) /p$ are
the best constants.

(ii) If $34/35<q\leq 1$ and $p\geq \pi ^{2}/4-1$, then the double
inequalities (\ref{M1p>0}) is reversed.

(iii) If $0<q\leq 34/35$ and $p\geq 3q-8/5$, then all the double
inequalities (\ref{M1p>0}), (\ref{M1p=0}) and (\ref{M1p<0}) are reversed.

(iv) If $q\leq 0$ and $p\geq 3q-8/5$, then the inequalities%
\begin{eqnarray*}
\left( \frac{\sin x}{x}\right) ^{p} &>&1-\frac{p}{3q}+\frac{p}{3q}\cos ^{q}x%
\text{ if }p>0, \\
\frac{\sin x}{x} &>&\exp \frac{\cos ^{q}x-1}{3q}\text{ if }p=0, \\
\left( \frac{\sin x}{x}\right) ^{p} &<&1-\frac{p}{3q}+\frac{p}{3q}\cos ^{q}x%
\text{ if }p<0,
\end{eqnarray*}%
hold, where $1/3$ is the best constant.
\end{theorem}

\begin{proof}
We only prove (i), others can be proved in the same way. By part (i) of
Proposition \ref{P main1}, if $q\geq 1$ and $p\leq 3q-8/5$, then we get $%
T_{p,q}\left( 0^{+}\right) <T_{p,q}\left( x\right) <T_{p,q}\left( \pi
/2^{-}\right) $, that is,%
\begin{equation*}
\frac{1}{3}C_{q}\left( x\right) <S_{p}\left( x\right) <T_{p,q}\left( \pi
/2^{-}\right) C_{q}\left( x\right) ,
\end{equation*}%
which is equivalent to (\ref{M1p>0}), (\ref{M1p=0}) and (\ref{M1p<0}). This
completes the proof.
\end{proof}

\begin{remark}
Letting $p=q$ in Theorem \ref{MT1 general} yields Theorem Zhu \ref{Zhu}. It
can be seen that our result is a generalization of Zhu's \cite%
{Zhu.CMA.58(2009)}.
\end{remark}

\begin{remark}
\label{Remark T-M-N}For $0<a<b$ and $\left( w,q\right) \in \Omega
_{w,q}=\{q\geq 0,w\leq 1$ or $q\leq 0,w\geq 0\}$, we define $M_{q}$ by%
\begin{equation}
M_{q}(a,b;w)=(wa^{q}+\left( 1-w\right) b^{q})^{1/q}\text{ if }q\neq 0\text{
\ and \ }M_{0}(a,b;w)=a^{w}b^{1-w}.  \label{M_q}
\end{equation}%
It is clear that $M_{q}(a,b;w)$ is a weighted power mean of order $q$ of $a$
and $b$ for $w\in \left( 0,1\right) $, but not a mean of $a$ and $b$ for $%
w\geq 1,q\leq 0$ or $w\leq 0,q\geq 0$ because that%
\begin{equation*}
M_{q}(a,b;w)\leq a\text{ for }w\geq 1,q\leq 0\text{ and}\ M_{q}(a,b;w)\geq b%
\text{ for }w\leq 0,q\geq 0.
\end{equation*}

When $w=p/\left( 3q\right) ,1-\left( 2/\pi \right) ^{p}$, that $\left(
w,q\right) \in \Omega _{w,q}$ implies that%
\begin{equation*}
\left( p,q\right) \in E_{p,q}=\{p\leq 0\text{ or }0<p\leq 3q\}\text{ \ and \ 
}\left( p,q\right) \in \mathbb{R}\times \mathbb{R}_{+},
\end{equation*}%
respectively. Thus, inequalities (\ref{M1p>0}), (\ref{M1p=0}) and (\ref%
{M1p<0}) can be unified into one:%
\begin{equation}
M_{q}^{q/p}\left( \cos x,1;1-\left( 2/\pi \right) ^{p}\right) <\frac{\sin x}{%
x}<M_{q}^{q/p}(\cos x,1;p/\left( 3q\right) ).  \label{N<sinx/x<M}
\end{equation}%
For convenience, we denote by%
\begin{equation*}
M_{q}^{q/p}(\cos x,1;p/\left( 3q\right) )=M(\cos x;p.q)\text{ \ and \ }%
M_{q}^{q/p}\left( \cos x,1;1-\left( 2/\pi \right) ^{p}\right) =N(\cos x;p.q).
\end{equation*}%
Then,%
\begin{equation}
M(t;p,q)=\left\{ 
\begin{array}{ll}
\left( 1-\frac{p}{3q}+\frac{p}{3q}t^{q}\right) ^{1/p} & \text{if }pq\neq
0,\left( p,q\right) \in E_{p,q}, \\ 
\exp \frac{t^{q}-1}{3q} & \text{if }p=0,q\neq 0, \\ 
\left( \frac{p}{3}\ln t+1\right) ^{1/p} & \text{if }p<0,q=0, \\ 
t^{1/3} & \text{if }p=q=0,%
\end{array}%
\right.  \label{M}
\end{equation}%
\begin{equation}
N(t;p,q)=\left\{ 
\begin{array}{ll}
\left( \left( \frac{2}{\pi }\right) ^{p}+\left( 1-\left( \frac{2}{\pi }%
\right) ^{p}\right) t^{q}\right) ^{1/p} & \text{if }p\neq 0, \\ 
\left( \frac{2}{\pi }\right) ^{1-t^{q}} & \text{if }p=0.%
\end{array}%
\right.  \label{N}
\end{equation}

And then, we have the following assertions:

(i) For $x\in (0,\pi /2)$, if $\left( p,q\right) \in E_{p,q}=\{\left(
p,q\right) :p\leq 0$ or $0<p\leq 3q\}$, then 
\begin{equation*}
T_{p,q}\left( x\right) >\left( <\right) T_{p,q}\left( 0^{+}\right)
\Longleftrightarrow \frac{\sin x}{x}<\left( >\right) M\left( \cos
x;p,q\right) .
\end{equation*}

(ii) For $x\in (0,\pi /2)$, if $\left( p,q\right) \in \mathbb{R}\times 
\mathbb{R}_{+}$, then 
\begin{equation*}
T_{p,q}\left( x\right) <\left( >\right) T_{p,q}\left( \pi /2^{-}\right)
\Longleftrightarrow \frac{\sin x}{x}>\left( <\right) N\left( \cos
x;p,q\right) .
\end{equation*}
\end{remark}

\begin{remark}
\label{Remark M,N}(i) For the monotonicity of $M\left( t;p,q\right) $ with
respect to $p,q$, we suggest that:

Let $E_{p,q}=\{\left( p,q\right) :p\leq 0$ or $0<p\leq 3q\}$ and let $M$ be
the function defined on $(0,1)\times E_{p,q}$ by (\ref{M}). Then $M$ is
decreasing in $p$ and increasing in $q$.

Indeed, for $pq\neq 0$, logarithmic differentiation yields%
\begin{eqnarray*}
\frac{\partial \ln M}{\partial p} &=&\frac{1}{p^{2}}\left( -\ln \left( \frac{%
p}{3q}t^{q}+1-\frac{p}{3q}\right) -\frac{p\left( 1-t^{q}\right) }{\left(
pt^{q}+3q-p\right) }\right) :=\frac{1}{p^{2}}M_{1}\left( t;p,q\right) , \\
\frac{\partial M_{1}}{\partial p} &=&-\frac{p\left( 1-t^{q}\right) ^{2}}{%
\left( pt^{q}+3q-p\right) ^{2}},
\end{eqnarray*}%
which implies that $M_{1}$ is decreasing in $p$ on $(0,\infty )$ and
increasing on $\left( -\infty ,0\right) $. Hence we have $M_{1}\left(
t;p,q\right) <M_{1}\left( t;0,q\right) =0$, and so $M$ is decreasing in $p$.

In the case of $pq\neq 0,\left( p,q\right) \in E_{p,q}$, it can be proved in
the same way.

Similarly, we have%
\begin{equation*}
\frac{\partial \ln M}{\partial q}=-\frac{t^{q}}{3q^{2}\left( \frac{p}{3q}%
t^{q}+1-\frac{p}{3q}\right) }\left( \ln t^{-q}-\left( t^{-q}-1\right)
\right) >0,
\end{equation*}%
where the last inequality holds due to $\ln x\leq x-1$ for $x>0$ and $\left(
\left( p/\left( 3q\right) \right) t^{q}+1-\left( p/\left( 3q\right) \right)
\right) >0$ for $\left( t,p,q\right) \in (0,1)\times E_{p,q}$, which proves
the monotonicity of $M$ with respect to $q$.

(ii) For the monotonicity of $N\left( t;p,q\right) $ with respect to $p,q$,
we claim that:

Let $N$ be defined on $\left( 0,1\right) \times \mathbb{R}\times \mathbb{R}%
_{+}$ by (\ref{N}).\ Then $N$ is increasing in $p$ and decreasing in $q$.

In fact, $N\left( t;p,q\right) $ can be written as%
\begin{equation*}
N\left( t;p,q\right) =\left( t^{q}+\left( 1-t^{q}\right) \left( \frac{2}{\pi 
}\right) ^{p}\right) ^{1/p}\text{ if }p\neq 0\text{ and }N\left(
t;0,q\right) =\left( \frac{2}{\pi }\right) ^{1-t^{q}},
\end{equation*}%
which is clearly a weighted power mean of order $p$ of positive numbers $1$
and $2/\pi $, and consequently, $N$ is increasing with respect to $p$.

The decreasing property of $N$ in $q$ can be derived from that for $p\neq 0$,%
\begin{equation*}
\frac{\partial N}{\partial q}=\left( t^{q}+\left( 1-t^{q}\right) \left( 
\frac{2}{\pi }\right) ^{p}\right) ^{1/p-1}\times \frac{1-\left( \frac{2}{\pi 
}\right) ^{p}}{p}\times t^{q}\ln t<0,
\end{equation*}
where the inequality is valid because that $\left( 1-\left( 2/\pi \right)
^{p}\right) /p>0$ and $t\in \left( 0,1\right) $.
\end{remark}

Utilizing Proposition \ref{P main3}, the following theorem is immediate.

\begin{theorem}
\label{MT2 generala}Let $x\in (0,\pi /2)$. Then

(i) if $p\geq \pi ^{2}/4-1$, then the double inequality%
\begin{equation}
\frac{1-\left( \cos x\right) ^{q_{2}}}{3q_{2}}<\frac{1-\left( \frac{\sin x}{x%
}\right) ^{p}}{p}<\frac{1-\left( \cos x\right) ^{q_{1}}}{3q_{1}}  \label{M2a}
\end{equation}%
holds for $q_{2}\geq p/3+8/15$ and $q_{1}\leq 1$;

(ii) if $p\in \lbrack 7/5,\pi ^{2}/4-1)$, then (\ref{M2a}) holds for $%
q_{2}\geq p/3+8/15$ and $q_{1}\leq 34/35$;

(iii) if $p\in \lbrack 46/35,7/5)$, then (\ref{M2a}) holds for $q_{2}\geq 1$
and $q_{1}\leq 34/35$;

(iv) if $p<46/35$, then (\ref{M2a}) holds for $q_{2}\geq 1$ and for $%
q_{1}\leq p/3+8/15$.
\end{theorem}

\begin{remark}
\label{Remark MT2changed}If $\left( p,q\right) \in E_{p,q}=\{\left(
p,q\right) :p\leq 0$ or $0<p\leq 3q\}$, then Theorem \ref{MT2 generala} can
be restated as follows: Let $x\in (0,\pi /2)$. Then

(i) if $p\geq \pi ^{2}/4-1$, then the double inequality%
\begin{equation}
\left( 1-\frac{p}{3q_{1}}+\frac{p}{3q_{1}}\cos ^{q_{1}}x\right) ^{1/p}<\frac{%
\sin x}{x}<\left( 1-\frac{p}{3q_{2}}+\frac{p}{3q_{2}}\cos ^{q_{2}}x\right)
^{1/p}  \label{M2b}
\end{equation}%
holds for $p/3\leq q_{1}\leq 1$ and $q_{2}\geq p/3+8/15$;

(ii) if $p\in \lbrack 7/5,\pi ^{2}/4-1)$, then (\ref{M2b}) holds for $%
p/3\leq q_{1}\leq 34/35$ and $q_{2}\geq p/3+8/15$;

(iii) if $p\in (46/35,7/5)$, then (\ref{M2b}) holds for $p/3\leq q_{1}\leq
34/35$ and $q_{2}\geq 1$;

(iv) if $p\in (0,46/35]$, then (\ref{M2b}) holds for $p/3\leq q_{1}\leq
p/3+8/15$ and $q_{2}\geq 1$;

(v) if $p\leq 0$, then (\ref{M2b}) holds for $q_{1}\leq p/3+8/15$ and $%
q_{2}\geq 1$.
\end{remark}

Before showing the sharp inequalities for trigonometric functions,, we give
a useful lemma.

\begin{lemma}
\label{Lemma E D_p.q=0}Let $q>0$ and let $D_{p,q}$ be defined on $(0,\pi /2)$
by 
\begin{equation}
D_{p,q}\left( x\right) =T_{p,q}\left( x\right) -\frac{1}{3}=\frac{%
S_{p}\left( x\right) }{C_{q}\left( x\right) }-\frac{1}{3}.  \label{D_p,q}
\end{equation}%
(i) We have%
\begin{eqnarray}
\lim_{x\rightarrow 0^{+}}\frac{D_{p,q}\left( x\right) }{x^{2}} &=&-\frac{1}{%
36}\left( p-3q+\frac{8}{5}\right) ,  \label{Limit1} \\
\lim_{x\rightarrow 0^{+}}\frac{D_{3q-8/5,q}\left( x\right) }{x^{4}} &=&\frac{%
1}{135}\left( q-\frac{34}{35}\right) ,  \label{Limit1s} \\
D_{p,q}\left( \pi /2^{-}\right) &=&\left\{ 
\begin{array}{ll}
\frac{q}{p}\left( 1-\left( \frac{2}{\pi }\right) ^{p}\right) -\frac{1}{3} & 
\text{if }q>0,p\neq 0, \\ 
-q\ln \frac{2}{\pi }-\frac{1}{3} & \text{if }q>0,p=0, \\ 
-\frac{1}{3} & \text{if }q\leq 0.%
\end{array}%
\right.  \label{Limit2}
\end{eqnarray}%
(ii) If $q>0$, then for fixed $p>0$, the equation $D_{p,q}\left( \pi
/2^{-}\right) =0$\ has a unique root $q(p)$ on $\mathbb{R}$ such that $%
D_{p,q}\left( \pi /2^{-}\right) >0$ for $q>q\left( p\right) $ and $%
D_{p,q}\left( \pi /2^{-}\right) <0$ for $q<q\left( p\right) $, where%
\begin{equation}
q\left( p\right) =\frac{p}{3\left( 1-\left( 2/\pi \right) ^{p}\right) }\text{
if }p\neq 0\text{ \ and \ }q\left( 0\right) =-\frac{1}{3\ln \left( 2/\pi
\right) }.  \label{q(p)}
\end{equation}%
(iii) For fixed $q>0$, the equation $D_{p,q}\left( \pi /2^{-}\right) =0$\
has a unique root $p\left( q\right) $ on $\mathbb{R}$ such that $%
D_{p,q}\left( \pi /2^{-}\right) >0$ for $p<p\left( q\right) $ and $%
D_{p,q}\left( \pi /2^{-}\right) <0$ for $p>p\left( q\right) $, where $%
p\left( q\right) $ is the inverse function of $q\left( p\right) $. In
particular, we have%
\begin{equation}
p_{0}=p\left( 1\right) \approx 1.42034\text{ \ and \ }p_{0}^{\ast }=p\left( 
\tfrac{34}{35}\right) \approx 1.27754;  \label{p_0,p_0*}
\end{equation}

(iv) Both functions $q\mapsto p\left( q\right) $ and $p\mapsto q\left(
p\right) $ are increasing.
\end{lemma}

\begin{proof}
(i) For $pq\neq 0$, expanding in power series yields%
\begin{equation*}
D_{p,q}\left( x\right) =-\frac{5p-15q+8}{180}x^{2}+\frac{%
70p^{2}+315q^{2}-315pq+126p+126q-304}{45360}x^{4}+o\left( x^{6}\right) ,
\end{equation*}%
which leads to (\ref{Limit1}). It is easy to check that it holds for $p=0$
or $q=0$.

If $p=3q-8/5$, then we have 
\begin{equation*}
D_{p,q}\left( x\right) =\frac{35q-34}{4725}x^{4}+o\left( x^{6}\right) ,
\end{equation*}%
which implies (\ref{Limit1s}).

(ii) If $q>0$, then for fixed $p>0$, solving the equation $D_{p,q}\left( \pi
/2^{-}\right) =0$ for $q$ we get $q=q\left( p\right) $, where $q\left(
p\right) $ is defined by (\ref{q(p)}). It is easy to check that $%
D_{p,q}\left( \pi /2^{-}\right) >0$ for $q>q\left( p\right) $ and $%
D_{p,q}\left( \pi /2^{-}\right) <0$ for $q<q\left( p\right) $.

(iii) For fixed $q>0$, by Lemma \ref{Lemma u_p}, we see that $p\mapsto
D_{p,q}\left( x\right) $ is decreasing on $\mathbb{R}$, which together with
the facts that%
\begin{equation*}
D_{-\infty ,q}\left( \frac{\pi }{2}^{-}\right) =\infty \text{ \ and \ }%
D_{\infty ,q}\left( \frac{\pi }{2}^{-}\right) =-\frac{1}{3}<0
\end{equation*}%
gives the desired assertion. Clearly, as a unique root of the equation $%
D_{p,q}\left( \pi /2^{-}\right) =0$, $p=p\left( q\right) $ is the inverse
function of $q\left( p\right) $. By mathematical computer software we can
find the approximations of $p\left( 1\right) $ and $p\left( 34/35\right) $.

(iv) From Lemma \ref{Lemma u_p} it is easy to see that $q\mapsto p\left(
q\right) $ is increasing, and so is its inverse.

This lemma is proved.
\end{proof}

Now we are ready to present sharp bounds for $\left( \sin x\right) /x$ in
terms of $M\left( \cos x;p,q\right) $ when $p$ is fixed.

\begin{theorem}
\label{MT2 generalc}Let $q\left( p\right) $ be defined by (\ref{q(p)}) and $%
\left( p,q\right) \in E_{p,q}=\{\left( p,q\right) :p\leq 0$ or $0<p\leq 3q\}$%
.

(i) If $p\geq 7/5$, then the inequality%
\begin{equation}
\frac{\sin x}{x}<\left( 1-\frac{p}{3q}+\frac{p}{3q}\cos ^{q}x\right) ^{1/p}
\label{M1uM<}
\end{equation}%
holds for $x\in (0,\pi /2)$ if and only if $q\geq p/3+8/15$.

(ii) If $p\geq p_{0}\approx 1.42034$, where $p_{0}$ is defined by (\ref%
{p_0,p_0*}), then the inequality%
\begin{equation}
\frac{\sin x}{x}>\left( 1-\frac{p}{3q}+\frac{p}{3q}\cos ^{q}x\right) ^{1/p}
\label{M1uM>}
\end{equation}%
holds for $x\in (0,\pi /2)$ if and only if $q\leq q\left( p\right) $;

(iii) if $p\leq 46/35$, then the inequality (\ref{M1uM>}) holds for $x\in
(0,\pi /2)$ if and only if $q\leq p/3+8/15$.

(iv) If $p\leq p_{0}^{\ast }\approx 1.27754$, where $p_{0}^{\ast }$ is
defined by (\ref{p_0,p_0*}), then (\ref{M1uM<}) holds for $x\in (0,\pi /2)$
if and only if and $q\geq q\left( p\right) $.
\end{theorem}

\begin{proof}
As shown in Remark \ref{Remark T-M-N}, if $\left( p,q\right) \in E_{p,q}$,
then inequality (\ref{M1uM<}) or (\ref{M1uM>}) holds if and only if $%
D_{p,q}\left( x\right) :=S_{p}\left( x\right) /C_{q}\left( x\right) -1/3>0$
(or $<0$).

(i) When $p\geq 7/5$, we prove the inequality (\ref{M1uM<}) holds for $x\in
(0,\pi /2)$ if and only if $q\geq p/3+8/15$. The necessity is obtained from%
\begin{equation*}
\lim_{x\rightarrow 0^{+}}x^{-2}D_{p,q}\left( x\right) =\frac{1}{24}\left(
q-\left( \frac{p}{3}+\frac{8}{15}\right) \right) \geq 0,
\end{equation*}%
which gives $q\geq p/3+8/15$. The sufficiency easily follows by parts (i)
and (ii) of Remark \ref{Remark MT2changed}.

(ii) When $p\geq p_{0}\approx 1.42034$ in which $p_{0}$ satisfies that $%
p_{0}\left( 1-\left( 2/\pi \right) ^{p_{0}}\right) ^{-1}/3=1$, we show that
the inequality (\ref{M1uM>}) holds for $x\in (0,\pi /2)$ if and only if $%
p/3\leq q\leq p\left( 1-\left( 2/\pi \right) ^{p}\right) ^{-1}/3$. The
necessity can be derived from $\lim_{x\rightarrow \pi /2^{-}}D_{p,q}\left(
x\right) \leq 0$, which by Lemma \ref{Lemma E D_p.q=0} yields $q\leq p\left(
1-\left( 2/\pi \right) ^{p}\right) ^{-1}/3=q\left( p\right) $.

Now we prove the sufficiency. Due to Remark \ref{Remark M,N}, it is seen
that $q\mapsto M\left( \cos x;p,q\right) $ is increasing, and it suffices to
show that the inequality (\ref{M1uM>}) holds for $x\in (0,\pi /2)$ when $%
q=q\left( p\right) $. Also, by part (iii) of Lemma \ref{Lemma E D_p.q=0}, $%
p\mapsto q\left( p\right) =p\left( 1-\left( 2/\pi \right) ^{p}\right)
^{-1}/3 $ is increasing, so we get $q=q\left( p\right) \geq q\left(
p_{0}\right) =p_{0}\left( 1-\left( 2/\pi \right) ^{p_{0}}\right) ^{-1}/3=1$.
From Lemma \ref{Lemma g1}, when $q\geq 1$, the function $g_{1}=\left(
q\left( p\right) B-C\right) /A$ is increasing on $(0,\pi /2)$, and so $%
x\mapsto p-g_{1}\left( x\right) :=h\left( x,p,q\left( p\right) \right) $ is
decreasing on $(0,\pi /2)$. But,%
\begin{eqnarray*}
h\left( 0^{+},p,q\left( p\right) \right) &=&p-\left( 3q\left( p\right) -%
\frac{8}{5}\right) =p-\frac{p}{1-\left( 2/\pi \right) ^{p}}+\frac{8}{5}, \\
h\left( \frac{\pi }{2}^{-},p,q\left( p\right) \right) &=&\left\{ 
\begin{array}{ll}
p-\infty & \text{if }q\left( p\right) >1, \\ 
p-\left( \frac{\pi ^{2}}{4}-1\right) <0 & \text{if }q\left( p\right) =1,%
\end{array}%
\right.
\end{eqnarray*}%
where $p-\left( \pi ^{2}/4-1\right) <0$ due to that $q\left( p\right) =1$
implies $p=p_{0}\approx 1.42034$. Also, we claim that $h\left(
0^{+},p,q\left( p\right) \right) >0$ for $p\geq p_{0}$. In fact,
differentiation leads to%
\begin{equation*}
h^{\prime }\left( 0^{+},p,q\left( p\right) \right) =-\frac{\left( 2/\pi
\right) ^{p}}{\left( 1-\left( 2/\pi \right) ^{p}\right) ^{2}}\left( \ln
\left( \frac{2}{\pi }\right) ^{p}-\left( \frac{2}{\pi }\right) ^{p}+1\right)
\geq 0,
\end{equation*}%
where the last inequality holds due to $\ln x\leq x-1$ for $x>0$. Hence, 
\begin{equation*}
h\left( 0^{+},p,q\left( p\right) \right) \geq h\left( 0^{+},p_{0,},q\left(
p_{0}\right) \right) =p_{0}-\frac{p_{0}}{1-\left( 2/\pi \right) ^{p_{0}}}+%
\frac{8}{5}=p_{0}-3+\frac{8}{5}>0.
\end{equation*}%
Therefore, there is a unique number $x_{0}\in \left( 0,\pi /2\right) $ such
that $h\left( x,p,q\left( p\right) \right) >0$ for $x\in \left(
0,x_{0}\right) $ and $h\left( x,p,q\left( p\right) \right) <0$ for $x\in
\left( x_{0},\pi /2\right) $, which together with (\ref{f2another}) and (\ref%
{df1}) means that the function $x\mapsto S_{p}^{\prime }\left( x\right)
/C_{q}^{\prime }\left( x\right) $ is decreasing on $(0,x_{0}]$\ and
increasing on $\left( x_{0},\pi /2\right) $. We note that $C_{q}^{\prime
}\left( x\right) =\cos ^{q-1}x\sin x>0$ for $x\in \left( 0,\pi /2\right) $
and the relation%
\begin{equation*}
\frac{S_{p}\left( \frac{\pi }{2}^{-}\right) -S_{p}\left( 0^{+}\right) }{%
C_{q}\left( \frac{\pi }{2}^{-}\right) -C_{q}\left( 0^{+}\right) }=\frac{1}{3}
\end{equation*}%
holds, utilizing Lemma \ref{Lemma Yang} it is derived that the inequality%
\begin{equation*}
\frac{S_{p}\left( x\right) -S_{p}\left( 0^{+}\right) }{C_{q}\left( x\right)
-C_{q}\left( 0^{+}\right) }<\frac{1}{3}
\end{equation*}%
holds for all $x\in \left( 0,\pi /2\right) $, that is, $S_{p}\left( x\right)
/C_{q}\left( x\right) <1/3$ is valid for $x\in \left( 0,\pi /2\right) $,
which prove the sufficiency.

(iii) When $p\leq 46/35$, we prove the inequality of (\ref{M1uM>}) holds $%
x\in (0,\pi /2)$ if and only if $q\leq p/3+8/15$ and $\left( p,q\right) \in
E_{p,q}$. The necessity easily follows from%
\begin{equation*}
\lim_{x\rightarrow 0^{+}}x^{-2}D_{p,q}\left( x\right) =\frac{1}{24}\left(
q-\left( \frac{p}{3}+\frac{8}{15}\right) \right) \leq 0,
\end{equation*}%
which gives $q\leq p/3+8/15$. The sufficiency easily follows by part (iv)
and (v) of Remark \ref{Remark MT2changed}.

(iv) Finally, we prove that when $p\leq p_{0}^{\ast }\approx 1.27754$ in
which $p_{0}^{\ast }$ satisfies that $p_{0}^{\ast }\left( 1-\left( 2/\pi
\right) ^{p_{0}^{\ast }}\right) ^{-1}/3=34/35$, the inequality (\ref{M1uM<})
holds for $x\in \left( 0,\pi /2\right) $ if and only if $q\geq q\left(
p\right) $. The necessity can be derived from $\lim_{x\rightarrow \pi
/2^{-}}D_{p,q}\left( x\right) \geq 0$, which by Lemma \ref{Lemma E D_p.q=0}
leads us to $q\geq q\left( p\right) $.

Similarly, to prove the sufficiency, it suffices to show that the inequality
(\ref{M1uM<}) holds for $x\in (0,\pi /2)$ when $q=q\left( p\right) $. Also,
by part (iii) of Lemma \ref{Lemma E D_p.q=0}, $p\mapsto q\left( p\right) $
is increasing, so we get $q\left( p\right) \leq q\left( p_{0}^{\ast }\right)
=p_{0}^{\ast }\left( 1-\left( 2/\pi \right) ^{p_{0}^{\ast }}\right)
^{-1}/3=34/35$. By Lemma \ref{Lemma g1}, when $q\leq 34/35$, the function $%
g_{1}=\left( qB-C\right) /A$ is decreasing on $(0,\pi /2)$, and so $x\mapsto
p-g_{1}\left( x\right) :=h\left( x,p,q\left( p\right) \right) $ is
increasing on $(0,\pi /2)$. But,%
\begin{eqnarray*}
h\left( 0^{+},p,q\left( p\right) \right) &=&p-\left( 3q\left( p\right) -%
\frac{8}{5}\right) =p-\frac{p}{1-\left( 2/\pi \right) ^{p}}+\frac{8}{5}, \\
h\left( \frac{\pi }{2}^{-},p,q\left( p\right) \right) &=&p+\infty .
\end{eqnarray*}%
As shown previously, $p\mapsto h\left( 0^{+},p,q\left( p\right) \right) $ is
increasing, and so for $p\leq p_{0}^{\ast }$%
\begin{equation*}
h\left( 0^{+},p,q\left( p\right) \right) \leq h\left( 0^{+},p_{0}^{\ast
},q\left( p_{0}^{\ast }\right) \right) =p_{0}^{\ast }-\left( 3\times \frac{34%
}{35}-\frac{8}{5}\right) =p_{0}^{\ast }-\frac{46}{35}<0,
\end{equation*}%
Thus, there is a unique number $x_{1}\in \left( 0,\pi /2\right) $ such that $%
h\left( x,p,q\left( p\right) \right) <0$ for $x\in \left( 0,x_{1}\right) $
and $h\left( x,p,q\left( p\right) \right) >0$ for $x\in \left( x_{1},\pi
/2\right) $, which together with (\ref{f2another}) and (\ref{df1}) means
that the function $x\mapsto S_{p}^{\prime }\left( x\right) /C_{q}^{\prime
}\left( x\right) $ is increasing on $(0,x_{1}]$\ and decreasing on $\left(
x_{1},\pi /2\right) $. Similar to part (ii) of this proof, utilizing Lemma %
\ref{Lemma Yang} again we see that the inequality $S_{p}\left( x\right)
/C_{q}\left( x\right) >1/3$ holds true for $x\in \left( 0,\pi /2\right) $,
which prove the sufficiency.

This completes the proof of this theorem.
\end{proof}

Letting $p=7/5$ and $p=p_{0}\approx 1.42034$ in Theorem \ref{MT2 generalc}
we have

\begin{corollary}
For $x\in (0,\pi /2)$, the double inequality%
\begin{equation*}
\left( 1-\frac{p_{0}}{3q_{1}}+\frac{p_{0}}{3q_{1}}\cos ^{q_{1}}x\right)
^{1/p_{0}}<\frac{\sin x}{x}<\left( 1-\frac{7}{15q_{2}}+\frac{7}{15q_{2}}\cos
^{q_{2}}x\right) ^{5/7}
\end{equation*}%
holds if and only if $p_{0}/3\leq q_{1}\leq 1$ and $q_{2}\geq 1$, where $%
p_{0}\approx 1.42034$ satisfies that $p_{0}\left( 1-\left( 2/\pi \right)
^{p_{0}}\right) ^{-1}/3=1$. Particularly, taking $q_{1}=q_{2}=1$ we have%
\begin{equation*}
\left( \frac{3-p_{0}}{3}+\frac{p_{0}}{3}\cos x\right) ^{1/p_{0}}<\frac{\sin x%
}{x}<\left( \frac{8}{15}+\frac{7}{15}\cos x\right) ^{5/7},
\end{equation*}%
where $p_{0}\approx 1.42034$ and $7/5=1.4$ are the best.
\end{corollary}

Letting $p=46/35$ and $p=p_{0}^{\ast }\approx 1.27754$ in Theorem \ref{MT2
generalc} we get

\begin{corollary}
For $x\in (0,\pi /2)$, the double inequality%
\begin{equation*}
\left( 1-\frac{46}{105q_{1}}+\frac{46}{105q_{1}}\cos ^{q_{1}}x\right)
^{35/46}<\frac{\sin x}{x}<\left( 1-\frac{p_{0}^{\ast }}{3q_{2}}+\frac{%
p_{0}^{\ast }}{3q_{2}}\cos ^{q_{2}}x\right) ^{1/p_{0}^{\ast }}
\end{equation*}%
holds if and only if $46/105\leq q_{1}\leq 34/35$ and $q_{2}\geq 34/35$,
where $p_{0}^{\ast }\approx 1.27754$ satisfies that $p_{0}^{\ast }\left(
1-\left( 2/\pi \right) ^{p_{0}^{\ast }}\right) ^{-1}/3=34/35$. Particularly,
putting $q_{1}=q_{2}=34/35$ we have%
\begin{equation*}
\left( \frac{28}{51}+\frac{23}{51}\cos ^{34/35}x\right) ^{35/46}<\frac{\sin x%
}{x}<\left( \frac{102-35p_{0}^{\ast }}{102}+\frac{35p_{0}^{\ast }}{102}\cos
^{34/35}x\right) ^{1/p_{0}^{\ast }},
\end{equation*}%
where $46/35\approx 1.3143$ and $p_{0}^{\ast }\approx 1.27754$ are the best.
\end{corollary}

Letting $p=0,1$ in Theorem \ref{MT2 generalc} we get

\begin{corollary}
(i) The double inequality%
\begin{equation*}
\exp \left( \frac{\cos ^{q_{1}}x-1}{3q_{1}}\right) <\frac{\sin x}{x}<\exp
\left( \frac{\cos ^{q_{2}}x-1}{3q_{2}}\right)
\end{equation*}%
holds for $x\in (0,\pi /2)$ if and only if $q_{1}\leq 8/15$ and $q_{2}\geq
\left( 3\ln \left( \pi /2\right) \right) ^{-1}\approx 0.73814$.

(ii) The double inequality%
\begin{equation}
1-\frac{1}{3q_{1}}+\frac{1}{3q_{1}}\cos ^{q_{1}}x<\frac{\sin x}{x}<1-\frac{1%
}{3q_{2}}+\frac{1}{3q_{2}}\cos ^{q_{2}}x  \label{M1uM>p=1}
\end{equation}%
holds for $x\in (0,\pi /2)$ if and only if $1/3\leq q_{1}\leq 13/15$ and $%
q_{2}\geq \pi /\left( 3\pi -6\right) \approx 0.91731$.
\end{corollary}

\begin{remark}
Letting $q_{1}=13/15,2/3,1/2,1/3$ and $q_{2}=1$ and using the increasing
property of $M\left( \cos x;p,q\right) $ in $q$, we get the following chain
of inequalities from (\ref{M1uM>p=1}):%
\begin{eqnarray*}
\cos ^{1/3}x &<&\tfrac{1}{3}+\tfrac{2}{3}\cos ^{1/2}x<\tfrac{1}{2}+\tfrac{1}{%
2}\cos ^{2/3}x \\
&<&\tfrac{8}{13}+\tfrac{5}{13}\cos ^{13/15}x<\frac{\sin x}{x}<\tfrac{2}{3}+%
\tfrac{1}{3}\cos x.
\end{eqnarray*}
\end{remark}

\begin{theorem}
\label{MT3 general}Let $p\left( q\right) $ be the unique root of equation $%
D_{p,q}\left( \pi /2^{-}\right) =0$ for fixed $q>0$ and $\left(
p_{i},q\right) \in E_{p_{i},q}=\{\left( p_{i},q\right) :p_{i}\leq 0$ or $%
0<p_{i}\leq 3q\}$, $i=1,2$.

(i) If $q\geq 1$, then the double inequality%
\begin{equation}
\left( 1-\frac{p_{1}}{3q}+\frac{p_{1}}{3q}\cos ^{q}x\right) ^{1/p_{1}}<\frac{%
\sin x}{x}<\left( 1-\frac{p_{2}}{3q}+\frac{p_{2}}{3q}\cos ^{q}x\right)
^{1/p_{2}}  \label{M3}
\end{equation}%
holds for $x\in (0,\pi /2)$ if and only if $p_{1}\geq p\left( q\right) $ and 
$p_{2}\leq 3q-8/5$.

(ii) If $0<q\leq 34/35$, then the double inequality (\ref{M3}) holds if and
only if $p_{1}\geq 3q-8/5$ and $p_{2}\leq p\left( q\right) $.

(iii) If $q\leq 0$, then the first inequality in (\ref{M3}) holds if and
only if $p_{1}\geq 3q-8/5$.
\end{theorem}

\begin{proof}
In the case of $q\geq 1$. For $\left( p_{2},q\right) \in E_{p_{2},q}$, the
second inequality in (\ref{M3}) is equivalent to $D_{p_{2},q}\left( x\right)
:=S_{p_{2}}\left( x\right) /C_{q}\left( x\right) -1/3>0$ for $x\in (0,\pi
/2) $. If it holds for all $x\in (0,\pi /2)$, then we have%
\begin{equation*}
\lim_{x\rightarrow 0^{+}}x^{-2}D_{p_{2},q}\left( x\right) =\frac{1}{24}%
\left( q-\left( \frac{p_{2}}{3}+\frac{8}{15}\right) \right) \geq 0,
\end{equation*}%
which yields $p_{2}\leq 3q-8/5$. Conversely, if $p_{2}\leq 3q-8/5$, then by
part (i) of Proposition \ref{P main1} we get that $T_{p_{2},q}=S_{p_{2}}%
\left( x\right) /C_{q}\left( x\right) $ is increasing on $(0,\pi /2)$, and
so $S_{p_{2}}\left( x\right) /C_{q}\left( x\right) >S_{p_{2}}\left(
0^{+}\right) /C_{q}\left( 0^{+}\right) =1/3$, which implies the second
inequality in (\ref{M3}).

For $\left( p_{1},q\right) \in E_{p_{1},q}$, if the first inequality in (\ref%
{M3}) holds for all $x\in (0,\pi /2)$, that is, $D_{p_{1},q}\left( x\right)
:=S_{p_{1}}\left( x\right) /C_{q}\left( x\right) -1/3<0$, then there must be 
$D_{p_{1},q}\left( \pi /2^{-}\right) \leq 0$, and from Lemma \ref{Lemma E
D_p.q=0} we get $p_{1}\geq p\left( q\right) $, where $p\left( q\right) $ is
the inverse function of $q\left( p\right) $ defined by (\ref{q(p)}).

Now we prove the condition $p_{1}\geq p\left( q\right) $ is sufficient for
the first inequality in (\ref{M3}) to hold. Lemma \ref{Lemma E D_p.q=0} tell
us that $q\mapsto p\left( q\right) $ is increasing, which together with $%
q\geq 1$ gives $p_{1}\geq p\left( q\right) \geq p\left( 1\right) =p_{0}$.
And, $p_{1}>p\left( q\right) $ means $q<q\left( p_{1}\right) $. From part
(ii) of Theorem \ref{MT2 generalc}, the first inequality in (\ref{M3}) holds
for $x\in \left( 0,\pi /2\right) $.

In the cases of $0<q\leq 34/35$ and $q\leq 0$, it can be proved in the same
method, here we omit details of proof.

This completes the proof.
\end{proof}

Letting $q=1,34/35$ in Theorem \ref{MT3 general}, we get

\begin{corollary}
(i) The double inequality%
\begin{equation}
\left( 1-\frac{p_{1}}{3}+\frac{p_{1}}{3}\cos x\right) ^{1/p_{1}}<\frac{\sin x%
}{x}<\left( 1-\frac{p_{2}}{3}+\frac{p_{2}}{3}\cos x\right) ^{1/p_{2}}
\label{M3q=1}
\end{equation}%
holds for $x\in (0,\pi /2)$ if and only if $1.42034\approx p_{0}\leq
p_{1}\leq 3$ and $p_{2}\leq 7/5$.

(ii) The double inequality%
\begin{equation*}
\left( 1-\frac{35p_{1}}{102}+\frac{35p_{1}}{102}\cos ^{34/35}x\right)
^{1/p_{1}}<\frac{\sin x}{x}<\left( 1-\frac{35p_{2}}{102}+\frac{35p_{2}}{102}%
\cos ^{34/35}x\right) ^{1/p_{2}}
\end{equation*}%
holds for $x\in (0,\pi /2)$ if and only if $46/35\leq p_{1}\leq 102/35$ and $%
p_{2}\leq p\left( 34/35\right) =p_{0}^{\ast }\approx 1.27754$.
\end{corollary}

\begin{remark}
Letting $p_{1}=3/2,2,3$ and $p_{2}=7/5,6/5,1$ in (\ref{M3q=1}) and using the
decreasing property of $M\left( t;p,q\right) $ in $p$, we get%
\begin{eqnarray*}
\left( \cos x\right) ^{1/3} &<&\left( \frac{1}{3}+\frac{2}{3}\cos x\right)
^{1/2}<\left( \frac{1}{2}+\frac{1}{2}\cos x\right) ^{2/3}<\frac{\sin x}{x} \\
&<&\left( \frac{8}{15}+\frac{7}{15}\cos x\right) ^{5/7}<\left( \frac{3}{5}+%
\frac{2}{5}\cos x\right) ^{5/6}<\frac{2}{3}+\frac{1}{3}\cos x.
\end{eqnarray*}
\end{remark}

Letting $q=0$ in Theorem \ref{MT3 general} we get

\begin{corollary}
For $x\in (0,\pi /2)$, the inequality%
\begin{equation*}
\frac{\sin x}{x}>\left( 1+\frac{p}{3}\ln \left( \cos x\right) \right) ^{1/p}
\end{equation*}%
holds if and only if $p\in \lbrack -8/5,0)$.
\end{corollary}

\subsection{In the case of $p=kq$}

Letting $p/q=k$. Then $E_{p,q}=\{\left( p,q\right) :p\leq 0$ or $0<p\leq
3q\} $ is changed into 
\begin{equation*}
E_{kq,q}=\{\left( k,q\right) :q\leq 0,k\geq 0\text{ or }q\geq 0,k\leq 3\},
\end{equation*}%
while $M\left( t;p,q\right) $ can be expressed as 
\begin{equation*}
M\left( t;kq,q\right) =\left\{ 
\begin{array}{ll}
\left( 1-\frac{k}{3}+\frac{k}{3}t^{q}\right) ^{1/\left( kq\right) } & \text{%
if }kq\neq 0,\left( k,q\right) \in E_{kq,q}, \\ 
\exp \frac{t^{q}-1}{3q} & \text{if }k=0,q\neq 0, \\ 
t^{1/3} & \text{if }q=0.%
\end{array}%
\right.
\end{equation*}

\begin{remark}
Similar to the monotonicity of $M\left( t;p,q\right) $, we claim that $%
M\left( t;kq,q\right) $ is decreasing (increasing) in $q$ if $k>\left(
<\right) 3$, and is decreasing (increasing) in $k$ if $q>\left( <\right) 0$.

In fact, logarithmic differentiations gives%
\begin{eqnarray*}
\frac{\partial \ln M}{\partial q} &=&\frac{1}{q^{2}}\left( \frac{qt^{q}\ln t%
}{3-k+kt^{q}}-\frac{1}{k}\ln \left( 1-\frac{k}{3}+\frac{k}{3}t^{q}\right)
\right) :=\frac{M_{2}\left( t;k,q\right) }{q^{2}}, \\
\frac{\partial M_{2}}{\partial q} &=&\frac{t^{q}\ln ^{2}t}{\left(
3-k+kt^{q}\right) ^{2}}q\left( 3-k\right) ,
\end{eqnarray*}%
which means that $M_{2}$ is decreasing (increasing) in $q$ on $(0,\infty )$
and increasing (decreasing) on $\left( -\infty ,0\right) $ if $k>\left(
<\right) 3$. Hence we have $M_{2}\left( t;k,q\right) <\left( >\right)
M_{2}\left( t;k,0\right) =0$ if $k>\left( <\right) 3$, and then, $M$ is
decreasing (increasing) in $q$ for $k>\left( <\right) 3$.

Analogously, the monotonicity of $M\left( t;kq,q\right) $ with respect to $k$
easily follows from the following relations:%
\begin{eqnarray*}
\frac{\partial \ln M}{\partial k} &=&\frac{1}{k^{2}}\left( \frac{k}{q}\frac{%
t^{q}-1}{3-k+kt^{q}}-\frac{1}{q}\ln \left( 1-\frac{k}{3}+\frac{k}{3}%
t^{q}\right) \right) :=\frac{M_{3}\left( t;k,q\right) }{k^{2}}, \\
\frac{\partial M_{3}}{\partial k} &=&-\frac{k}{q}\frac{\left( t^{q}-1\right)
^{2}}{\left( 3-k+kt^{q}\right) ^{2}}.
\end{eqnarray*}
\end{remark}

As a direct consequence of Corollary \ref{Corollary p=kq}, we have

\begin{theorem}
\label{MT4 p=kq}Let $\left( k,q\right) \in E_{kq,q}=\{\left( k,q\right)
:q\leq 0,k\geq 0$ or $q\geq 0,k\leq 3\}$. Then

(i) when $k\in \left( 3,\infty \right) $, the inequality%
\begin{equation*}
\frac{\sin x}{x}>\left( 1-\frac{k}{3}+\frac{k}{3}\cos ^{q}x\right)
^{1/\left( kq\right) }
\end{equation*}%
holds for $x\in (0,\pi /2)$ if $8/\left( 5\left( 3-k\right) \right) \leq
q\leq 0$;

(ii) when $k\in \lbrack \left( 35\pi ^{2}-140\right) /136,3)$, the double
inequality%
\begin{equation}
\left( 1-\frac{k}{3}+\frac{k}{3}\cos ^{q_{1}}x\right) ^{1/\left(
kq_{1}\right) }<\frac{\sin x}{x}<\left( 1-\frac{k}{3}+\frac{k}{3}\cos
^{q_{2}}x\right) ^{1/\left( kq_{2}\right) }  \label{M4double}
\end{equation}%
holds for $x\in (0,\pi /2)$ if $q_{2}\geq 8/\left( 5\left( 3-k\right)
\right) $ and $q_{1}\leq 1$;

(iii) when $k\in \lbrack \pi ^{2}/4-1,\left( 35\pi ^{2}-140\right) /136)$,
the double inequality (\ref{M4double}) holds for $x\in (0,\pi /2)$ if $%
q_{2}\geq 8/\left( 5\left( 3-k\right) \right) $ and $q_{1}\leq 34/35$ or $%
\left( \pi ^{2}/4-1\right) /k\leq q_{1}\leq 1$;

(iv) when $k\in \lbrack 7/5,\pi ^{2}/4-1)$, the double inequality (\ref%
{M4double}) holds for $x\in (0,\pi /2)$ if $q_{2}\geq 8/\left( 5\left(
3-k\right) \right) $ and $q_{1}\leq 34/35$;

(v) when $k\in \lbrack 23/17,7/5)$, the double inequality (\ref{M4double})
holds for $x\in (0,\pi /2)$ if $q_{2}\geq 1$ and $q_{1}\leq 34/35$;

(vi) when $k\in \lbrack 0,23/17)$, the double inequality (\ref{M4double})
holds for $x\in (0,\pi /2)$ if $q_{2}\geq 1$ and $q_{1}\leq 8/\left( 5\left(
3-k\right) \right) $;

(vii) when $k\in \left( -\infty ,0\right) $, the double inequality (\ref%
{M4double}) holds for $x\in (0,\pi /2)$ if $q_{2}\geq 1$ and $0\leq
q_{1}\leq 8/\left( 5\left( 3-k\right) \right) $.
\end{theorem}

Now we give sharp bounds $M\left( \cos x;kq,q\right) $ for $\left( \sin
x\right) /x$ in the case of $k\in \left( 0,3\right) $. To this end, we note
that $D_{kq,q}\left( \pi /2^{-}\right) $ can be expressed as 
\begin{equation*}
D_{kq,q}\left( \pi /2^{-}\right) =\left\{ 
\begin{array}{ll}
\frac{1-\left( \frac{2}{\pi }\right) ^{kq}}{3q}-\frac{1}{3q} & \text{if }q>0,
\\ 
-\infty & \text{if }q\leq 0.%
\end{array}%
\right.
\end{equation*}%
It is easy to verify that there is a unique number $q\left( k\right) =\frac{%
\ln \left( 1-k/3\right) }{k\ln \left( 2/\pi \right) }$ such that $%
D_{kq,q}\left( \pi /2^{-}\right) >0$ for $q>q\left( k\right) $ and $%
D_{kq,q}\left( \pi /2^{-}\right) <0$ for $q<q\left( k\right) $. Also, we
easily see that $k\mapsto q\left( k\right) $ is increasing on $\left(
0,3\right) $, and $q\left( p_{0}\right) =1$, $q\left( p_{0}^{\ast }\right)
=34/35$. Meanwhile, $p=kq\leq \left( \geq \right) 3q-8/5$ implies that $%
q\geq \left( \leq \right) \frac{8}{5\left( 3-k\right) }$. Based on these
preparations above, using the same method of proof as Theorem \ref{MT2
generalc}'s, we can show the following theorem, whose proof is omitted.

\begin{theorem}
\label{MT4 p=kq sharp}Let $k\in \left( 0,3\right) $ and $x\in (0,\pi /2)$.
Then

(i) if $k\in \lbrack 7/5,3)$, then the inequality%
\begin{equation}
\frac{\sin x}{x}<\left( 1-\frac{k}{3}+\frac{k}{3}\cos ^{q}x\right)
^{1/\left( kq\right) }  \label{M4<}
\end{equation}%
holds if and only if $q\geq 8/\left( 5\left( 3-k\right) \right) $;

(ii) if $k\in \left( p_{0},3\right) $, then the inequality%
\begin{equation}
\left( 1-\frac{k}{3}+\frac{k}{3}\cos ^{q}x\right) ^{1/\left( kq\right) }<%
\frac{\sin x}{x}  \label{M4>}
\end{equation}%
holds if and only if $q\leq \left( \ln \left( 1-k/3\right) \right) /\left(
k\ln \left( 2/\pi \right) \right) $, where $p_{0}\approx 1.42034$ is defined
by (\ref{p_0,p_0*});

(iii) if $k\in (0,23/17]$, then the inequality (\ref{M4>}) holds if and only
if $q\leq 8/\left( 5\left( 3-k\right) \right) $;

(iv) if $k\in \left( 0,p_{0}^{\ast }\right) $, then the inequality (\ref{M4<}%
) holds if and only if $q\geq \left( \ln \left( 1-k/3\right) \right) /\left(
k\ln \left( 2/\pi \right) \right) $, where $p_{0}^{\ast }$ $\approx 1.27754$
is defined by (\ref{p_0,p_0*}).
\end{theorem}

Taking $k=1,3/2,2$ in Theorem \ref{MT4 p=kq}, we get immediately

\begin{corollary}
For $x\in (0,\pi /2)$, (i) the double inequality%
\begin{equation}
\left( \frac{2}{3}+\frac{1}{3}\cos ^{q_{1}}x\right) ^{1/q_{1}}<\frac{\sin x}{%
x}<\left( \frac{2}{3}+\frac{1}{3}\cos ^{q_{2}}x\right) ^{1/q_{2}}
\label{M4k=1}
\end{equation}%
holds if and only if $q_{1}\leq 4/5$ and $q_{2}\geq \left( \ln 3-\ln
2\right) /\left( \ln \pi -\ln 2\right) \approx 0.89788$;

(ii) the double inequality%
\begin{equation}
\left( \frac{1}{2}+\frac{1}{2}\cos ^{q_{1}}x\right) ^{2/\left( 3q_{1}\right)
}<\frac{\sin x}{x}<\left( \frac{1}{2}+\frac{1}{2}\cos ^{q_{2}}x\right)
^{2/\left( 3q_{2}\right) }  \label{M4k=3/2}
\end{equation}%
holds if and only if $q_{1}\leq \left( 2\ln 2\right) /\left( 3\left( \ln \pi
-\ln 2\right) \right) \approx 1.0233$ and $q_{2}\geq 16/15\approx 1.0667$;

(iii) the double inequality%
\begin{equation}
\left( \frac{1}{3}+\frac{2}{3}\cos ^{q_{1}}x\right) ^{1/\left( 2q_{1}\right)
}<\frac{\sin x}{x}<\left( \frac{1}{3}+\frac{2}{3}\cos ^{q_{2}}x\right)
^{1/\left( 2q_{2}\right) }  \label{M4k=2}
\end{equation}%
holds if and only if $q_{1}\leq \left( \ln 3\right) /\left( 2\left( \ln \pi
-\ln 2\right) \right) \approx 1.2164$ and $q_{2}\geq 8/5$.
\end{corollary}

\begin{remark}
Inequalities (\ref{M4k=1}) is exactly (\ref{Yang3}) given in \cite%
{Yang.MIA.17.2.2014}.
\end{remark}

\subsection{In the case of $p=3q-8/5$}

When $p=3q-8/5$, $E_{p,q}=\{\left( p,q\right) :p\leq 0$ or $0<p\leq 3q\}$ is
changed into%
\begin{equation}
E_{3q-8/5,q}=\{3q-8/5\leq 0\text{ or }0<3q-8/5\leq 3q\}=\mathbb{R}.
\label{Op=3q-8/5}
\end{equation}%
While $M\left( t;3q-8/5,q\right) $ can be completely written as%
\begin{equation}
M\left( t;3q-\tfrac{8}{5},q\right) =\left\{ 
\begin{array}{ll}
\left( \frac{8}{15q}+\left( 1-\frac{8}{15q}\right) t^{q}\right) ^{5/\left(
15q-8\right) } & \text{if }q\neq 0,\tfrac{8}{15}, \\ 
\left( 1-\frac{8}{15}\ln t\right) ^{-5/8} & \text{if }q=0, \\ 
\exp \frac{5\left( t^{8/15}-1\right) }{8} & \text{if }q=\tfrac{8}{15},%
\end{array}%
\right.  \label{Mp=3q=8/5}
\end{equation}%
where $t=\cos x\in \left( 0,1\right) $ for $x\in \left( 0,\pi /2\right) $. $%
N\left( t;3q-8/5,q\right) $ can be expressed as%
\begin{eqnarray*}
N\left( t;3q-\tfrac{8}{5},q\right) &=&\left( \left( \tfrac{2}{\pi }\right)
^{3q-8/5}+\left( 1-\left( \tfrac{2}{\pi }\right) ^{3q-8/5}\right)
t^{q}\right) ^{1/\left( 3q-8/5\right) }\text{if }q\neq \tfrac{8}{15},q>0 \\
N\left( t;0,\tfrac{8}{15}\right) &=&\left( \tfrac{2}{\pi }\right)
^{1-t^{8/15}}\text{ if }q=\tfrac{8}{15}\text{.}
\end{eqnarray*}%
For the monotonicity of $M\left( t;3q-8/5,q\right) $ in $q$, we can prove
the following

\begin{lemma}
\label{Lemma Mincreasing}Let $\left( t,q\right) \mapsto M\left(
t;3q-8/5,q\right) $ be defined on $\left( 0,1\right) \times \mathbb{R}$ by (%
\ref{Mp=3q=8/5}). Then $q\mapsto M\left( t;3q-8/5,q\right) $ is increasing
on $\mathbb{R}$, and we have%
\begin{equation}
\lim_{q\rightarrow -\infty }M\left( t;3q-8/5,q\right) =t^{1/3}\text{ \ and \ 
}\lim_{q\rightarrow \infty }M\left( t;3q-8/5,q\right) =1.  \label{Limit M}
\end{equation}
\end{lemma}

\begin{proof}
For $q\neq 0,8/15$, logarithmic differentiation yields 
\begin{eqnarray*}
\frac{\partial \ln M}{\partial q} &=&-3\frac{\ln \left( \frac{8}{15q}+\left(
1-\frac{8}{15q}\right) t^{q}\right) }{\left( 3q-\frac{8}{5}\right) ^{2}}-%
\frac{\frac{8}{15q^{2}}\left( 1-t^{q}\right) +t^{q}\left( \frac{8}{15q}%
-1\right) \ln t}{\left( \frac{8}{15q}+\left( 1-\frac{8}{15q}\right)
t^{q}\right) \left( 3q-\frac{8}{5}\right) }, \\
\frac{\partial \ln M}{\partial t} &=&5q\frac{t^{q-1}}{\left( 15q-8\right)
t^{q}+8}, \\
\frac{\partial ^{2}\ln M}{\partial q\partial t} &=&40t^{q}\frac{\ln
t^{q}-t^{q}+1}{t\left( \left( 15q-8\right) t^{q}+8\right) ^{2}}<0,
\end{eqnarray*}%
where the inequality holds due to $\ln x\leq x-1$ for $x>0$. Hence, $%
\partial \left( \ln M\right) /\partial q$ is decreasing in $t$, and so we
have%
\begin{equation*}
\frac{\partial \ln M}{\partial q}\left( t;3q-8/5,q\right) >\frac{\partial
\ln M}{\partial q}\left( 1;3q-8/5,q\right) =0,
\end{equation*}%
which means that $q\mapsto M\left( t;3q-8/5,q\right) $ has increasing
property. A direct computation yields (\ref{Limit M}).
\end{proof}

By Corollary \ref{Corollary p=3q-8/5} we immediately get

\begin{theorem}
\label{MT5 p=3q-8/5a}Let $x\in (0,\pi /2)$.

(i) If $q\geq 1$, then the double inequality%
\begin{equation}
N\left( \cos x,3q-\tfrac{8}{5},q\right) <\frac{\sin x}{x}<M\left( \cos x,3q-%
\tfrac{8}{5},q\right) \text{,}  \label{M5}
\end{equation}%
hold, where $M\left( t;3q-8/5,q\right) $ and $N\left( t;3q-8/5,q\right) $\
are defined by (\ref{M}) and (\ref{N}), respectively.

(ii) If $0<q\leq 34/35$, then the double inequality (\ref{M5}) is reversed.

(iii) If $q\leq 0$, then the inequality%
\begin{equation*}
\frac{\sin x}{x}>M\left( \cos x,3q-\tfrac{8}{5},q\right)
\end{equation*}%
holds.
\end{theorem}

In order to establish sharp inequalities $\left( \sin x\right) /x<\left(
>\right) M\left( \cos x;3q-8/5,q\right) $, the following lemma is also
needed.

\begin{lemma}
\label{Lemma E D_p.qs}Let $M\left( t;3q-8/5,q\right) $ be defined by (\ref%
{Mp=3q=8/5}). Then there is a unique $q_{0}\approx 0.989681$ such that $%
D_{3q-8/5,q}\left( \pi /2^{-}\right) >0$ for $q>q_{0}$ and $%
D_{3q-8/5,q}\left( \pi /2^{-}\right) <0$ for $q<q_{0}$.
\end{lemma}

\begin{proof}
From (\ref{Limit2}) we have%
\begin{equation}
D_{3q-8/5,q}\left( \pi /2^{-}\right) =\left\{ 
\begin{array}{ll}
\frac{1-\left( \frac{2}{\pi }\right) ^{3q-8/5}}{3q-8/5}-\frac{1}{3q} & \text{%
if }q>0,q\neq 8/15, \\ 
-\ln \frac{2}{\pi }-\frac{5}{8}<0 & \text{if }q=8/15, \\ 
-\infty & \text{if }q\leq 0\text{.}%
\end{array}%
\right.  \label{D_1q-8/5,q}
\end{equation}%
It is obvious that $D_{3q-8/5,q}\left( \pi /2^{-}\right) <0$ for $q\leq 8/15$%
, and it remains to show that $D_{3q-8/5,q}\left( \pi /2^{-}\right) <0$ for $%
q\in \left( 8/15,q_{0}\right) $ and $D_{3q-8/5,q}\left( \pi /2^{-}\right) >0$
for $q>q_{0}$. For $q>8/15$, we have%
\begin{equation*}
\left( 3q-8/5\right) D_{3q-8/5,q}\left( \pi /2^{-}\right) =\tfrac{8}{15q}%
-\left( \tfrac{2}{\pi }\right) ^{3q-8/5}=L\left( \tfrac{8}{15q},\left( 
\tfrac{2}{\pi }\right) ^{3q-8/5}\right) \times v\left( q\right) ,
\end{equation*}%
where $L\left( x,y\right) $ is the logarithmic mean of positive $x$ and $y$,
and%
\begin{equation*}
v\left( q\right) =\ln \tfrac{8}{15q}-\left( 3q-8/5\right) \ln \tfrac{2}{\pi }%
.
\end{equation*}%
Differentiation leads us to%
\begin{equation*}
v^{\prime }\left( q\right) =\frac{3}{q}\left( q-\frac{1}{3\ln \left( \pi
/2\right) }\right) \ln \frac{\pi }{2},
\end{equation*}%
which tell us that $v$ is increasing on $q\geq \left( 3\ln \left( \pi
/2\right) \right) ^{-1}\approx 0.73814$ and decreasing on $8/15<q\leq \left(
3\ln \left( \pi /2\right) \right) ^{-1}$. Hence, we have $v\left( q\right)
<v\left( 8/15\right) =0$ for $8/15<q\leq \left( 3\ln \left( \pi /2\right)
\right) ^{-1}$, but%
\begin{equation*}
v\left( 1\right) =\ln \frac{8}{15}-\frac{7}{5}\ln \frac{2}{\pi }\approx
3.6071\times 10^{-3}>0,
\end{equation*}%
and it follows that there is a unique number $q_{0}$ to satisfy $v\left(
q_{0}\right) =0$ such that $v\left( q\right) <0$ for $q\in \left(
8/15,q_{0}\right) $ and $v\left( q\right) >0$ for $q\in \left( q_{0},\infty
\right) $. Solving the equation $v\left( q_{0}\right) =0$ by mathematical
computer software we find that $q_{0}\approx 0.989681$.

This proves the lemma.
\end{proof}

\begin{theorem}
\label{MT5 p=3q-8/5b}Let $x\in (0,\pi /2)$. Then the inequality $\left( \sin
x\right) /x>M\left( \cos x;3q-8/5,q\right) $ holds if and only if $q\leq
34/35$, while its reverse holds if and only if $q\geq q_{0}\approx 0.989681$%
, where $q_{0}$ is the unique root of equation $D_{3q-8/5,q}\left( \pi
/2^{-}\right) =0$ on $\mathbb{R}_{+}$, here $D_{3q-8/5,q}\left( \pi
/2^{-}\right) $ is defined by (\ref{D_1q-8/5,q}).
\end{theorem}

\begin{proof}
Clearly, the inequality $\left( \sin x\right) /x<\left( >\right) M\left(
\cos x;3q-8/5,q\right) $ is equivalent to $D_{3q-8/5,q}\left( x\right)
:=S_{3q-8/5}\left( x\right) /C_{q}\left( x\right) -1/3>\left( <\right) 0$
for $x\in (0,\pi /2)$, where $M\left( t;3q-8/5,q\right) $ is defined by (\ref%
{Mp=3q=8/5}).

(i) We first prove that the inequality $\left( \sin x\right) /x>M\left( \cos
x;3q-8/5,q\right) $ if and only if $q\leq 34/35$. The necessity is due to $%
\lim_{x\rightarrow 0^{+}}x^{-4}D_{3q-8/5,q}\left( x\right) \leq 0$. By (\ref%
{Limit1s}) we get 
\begin{equation*}
\lim_{x\rightarrow 0^{+}}\frac{D_{3q-8/5,q}\left( x\right) }{x^{4}}=\frac{1}{%
135}\left( q-\frac{34}{35}\right) \leq 0,
\end{equation*}%
which yields $q\leq 34/35$. The sufficiency is obtained from Theorem \ref%
{MT5 p=3q-8/5a}.

(ii) Now we show that the inequality $\left( \sin x\right) /x<M\left( \cos
x;3q-8/5,q\right) $ if and only if $q\geq q_{0}\approx 0.989681$. The
necessity can be derived by the relation $D_{3q-8/5,q}\left( \pi
/2^{-}\right) =S_{p}\left( \pi /2^{-}\right) /C_{q}\left( \pi /2^{-}\right)
-1/3\geq 0$, which follows by part (ii) of Lemma \ref{Lemma E D_p.qs}.

Next we show the sufficiency. In the case of $q\geq 1$, it is obviously true
by Theorem \ref{MT5 p=3q-8/5a}. In the case of $q\in \lbrack p_{0},1)$, we
see that $f_{2}\left( x\right) $ defined by (\ref{f2}) can be written as%
\begin{equation}
f_{2}\left( x\right) =\left( 3q-8/5\right) A\left( x\right) -qB\left(
x\right) +C\left( x\right) =-\left( B\left( x\right) -3A\left( x\right)
\right) \left( q-g_{2}\left( x\right) \right) ,  \label{f2s}
\end{equation}%
where $g_{2}\left( x\right) $ is defined by (\ref{g2}) and $\left( B\left(
x\right) -3A\left( x\right) \right) >0$ for $x\in (0,\pi /2)$. By Lemma \ref%
{Lemma g2}, we see that $g_{2}=\left( C\left( x\right) -\left( 8/5\right)
A\left( x\right) \right) /\left( B\left( x\right) -3A\left( x\right) \right) 
$ is increasing on $(0,\pi /2)$, and so $x\mapsto q-g_{2}\left( x\right)
:=j\left( x\right) $ is decreasing on $(0,\pi /2)$. But,%
\begin{equation*}
j\left( 0^{+}\right) =q-34/35>0\text{ \ and }j\left( \frac{\pi }{2}%
^{-}\right) =q_{0}-1<0,
\end{equation*}%
then it is seen that there is a unique number $x_{3}\in (0,\pi /2)$ such
that $j\left( x\right) >0$ for $x\in \left( 0,x_{3}\right) $ and $j\left(
x\right) <0$ for $x\in \left( x_{3},\pi /2\right) $. This together with (\ref%
{f2s}) and $\left( B\left( x\right) -3A\left( x\right) \right) >0$ means
that $f_{2}\left( x\right) <0$ for $x\in \left( 0,x_{3}\right) $ and $%
f_{2}\left( x\right) >0$ for $x\in \left( x_{3},\pi /2\right) $. From the
relations (\ref{f2another}) and (\ref{df1}) it is derived that the function $%
x\mapsto S_{3q-8/5}^{\prime }\left( x\right) /C_{q}^{\prime }\left( x\right) 
$ is increasing on $(0,x_{3}]$\ and decreasing on $\left( x_{3},\pi
/2\right) $. Since $C_{q}^{\prime }\left( x\right) =\cos ^{q-1}x\sin x>0$
for $x\in \left( 0,\pi /2\right) $ and for $q\in \lbrack p_{0},1)$ the
relation%
\begin{equation*}
\frac{S_{3q-8/5}\left( \frac{\pi }{2}^{-}\right) -S_{3q-8/5}\left(
0^{+}\right) }{C_{q}\left( \frac{\pi }{2}^{-}\right) -C_{q}\left(
0^{+}\right) }>\frac{1}{3}
\end{equation*}%
holds, make use of Lemma \ref{Lemma Yang} it is deduced that the inequality%
\begin{equation*}
\frac{S_{3q-8/5}\left( x\right) -S_{3q-8/5}\left( 0^{+}\right) }{C_{q}\left(
x\right) -C_{q}\left( 0^{+}\right) }>\frac{1}{3}
\end{equation*}%
holds for all $x\in \left( 0,\pi /2\right) $, that is, $S_{3q-8/5}\left(
x\right) /C_{q}\left( x\right) >1/3$ is valid for $x\in \left( 0,\pi
/2\right) $ in the case of $q\in \lbrack p_{0},1)$. Thus the sufficiency
follows from Lemma \ref{Lemma Yang}.

This completes the proof of this theorem.
\end{proof}

We close this section by giving a very nice chain of inequalities for
trigonometric functions. Taking $q=-\infty
,0,8/15,7/10,4/5,13/15,34/35;q_{0},1,6/5$, we deduce that

\begin{corollary}
For $x\in (0,\pi /2)$, the chain of inequalities holds:%
\begin{eqnarray*}
\cos ^{1/3}x &<&\cdot \cdot \cdot <\left( 1-\frac{8}{15}\ln \left( \cos
x\right) \right) ^{-5/8}<\left( \frac{8}{3}-\frac{5}{3}\cos ^{1/5}x\right)
^{-1} \\
&<&\exp \left( \frac{5}{8}\cos ^{8/15}x-\frac{5}{8}\right) <\left( \frac{5}{%
21}\cos ^{7/10}x+\frac{16}{21}\right) ^{2} \\
&<&\left( \frac{1}{3}\cos ^{4/5}x+\frac{2}{3}\right) ^{5/4}<\frac{5}{13}\cos
^{13/15}x+\frac{8}{13} \\
&<&\left( \tfrac{23}{51}\cos ^{34/35}x+\tfrac{28}{51}\right) ^{35/46}<\frac{%
\sin x}{x}<\left( \frac{7}{15}\cos x+\frac{8}{15}\right) ^{5/7} \\
&<&\sqrt{\frac{5}{9}\cos ^{6/5}x+\frac{4}{9}}<\frac{2+\cos x}{3}.
\end{eqnarray*}
\end{corollary}

\begin{proof}
Employing Theorem \ref{MT5 p=3q-8/5b} and increasing property of $M\left(
\cos x;3q-8/5,q\right) $ with respect to $q$, we get the all inequalities
except for the last one. In order for the last one to be valid, it suffices
that the inequality%
\begin{equation*}
5\cos ^{6/5}x+4<\left( 2+\cos x\right) ^{2}
\end{equation*}%
holds for $x\in (0,\pi /2)$. With $\cos ^{1/5}x=t$, then we get%
\begin{equation*}
5t^{6}+4-\left( 2+t^{5}\right) ^{2}=-t^{5}\left( t-1\right) ^{2}\left(
t^{3}+2t^{2}+3t+4\right) <0,
\end{equation*}%
which proves the desired inequality.
\end{proof}

\section{Applications}

As demonstrated by Zhu in \cite{Zhu.CMA.58(2009)}, an inequality for
trigonometric functions can be changed into another one by making identical
transformations or changes of variables. Based on our results in previous
sections, we can obtain corresponding inequalities in the same way. For
example, multiplying the each sides in double inequality (\ref{M1p>0}) by $%
\left( \left( \sin x\right) /x\right) ^{-q}$ yields a new Huygens type one:%
\begin{equation*}
\left( \tfrac{2}{\pi }\right) ^{p}(\tfrac{x}{\sin x})^{q}+(1-\left( \tfrac{2%
}{\pi }\right) ^{p})\left( \tfrac{x}{\tan x}\right) ^{q}<\left( \tfrac{\sin x%
}{x}\right) ^{p-q}<(1-\tfrac{p}{3q})(\tfrac{x}{\sin x})^{q}+\tfrac{p}{3q}%
\left( \tfrac{x}{\tan x}\right) ^{q}
\end{equation*}%
if $q\geq 1$ and $0<p\leq 3q-8/5$. And then, Theorem \ref{MT1 general} can
be restated in Huygens type.

Next we will give some other applications.

\subsection{Shafer-Fink type and Carlson type inequalities}

In \cite[3, p. 247, 3.4.31]{Mitrinovic.AI.1970}, it was listed that the
inequality%
\begin{equation*}
\arcsin x>\frac{6\left( \sqrt{x+1}-\sqrt{1-x}\right) }{4+\sqrt{x+1}+\sqrt{1-x%
}}>\dfrac{3x}{2+\sqrt{1-x^{2}}}
\end{equation*}%
hold for $x\in \left( 0,1\right) $, which is due to Shafer \cite%
{Shafer.AMM.74.6.1967}. Fink \cite{Fink.UBPEF.6.1995} proved that the double
inequality%
\begin{equation*}
\tfrac{3x}{2+\sqrt{1-x^{2}}}\leq \arcsin x\leq \tfrac{\pi x}{2+\sqrt{1-x^{2}}%
}
\end{equation*}%
is true for $x\in \left[ 0,1\right] $. There has some improvements,
generalizations of Shafer-Fink inequality (see \cite{Zhu.MIA.8.4.2005}, \cite%
{Malesevic.JIA.2007.78691}, \cite{Malesevic.MIA.10.3.2007}, \cite%
{Zhu.JIA.2007.67430}, \cite{Guo.Filomat.27.2.2013}, \cite%
{Yang.AAA.2014.364076}).

Carlson \cite[(1.14)]{Carlson.PAMS.25.1970} inequality states that the
double inequality%
\begin{equation}
\frac{6\left( 1-t\right) ^{1/2}}{2\sqrt{2}+\left( 1+t\right) ^{1/2}}<\arccos
t<\frac{2^{2/3}\left( 1-t\right) ^{1/2}}{\left( 1+t\right) ^{1/6}}
\label{Carlson}
\end{equation}%
holds for $t\in \left( 0,1\right) $. As a corollary of Theorem Zhu given in 
\cite[Theorem 5]{Zhu.CMA.58(2009)}, he gave a generalization of (\ref%
{Carlson}).

Shafer-Fink type and Carlson inequalities are essentially attributed to ones
involving the functions $\left( \sin x\right) /x$ and $\cos x$, where $x\in
\left( 0,\pi /2\right) $. Therefore, after making a change of variable $\sin
x=t$ or $\cos x=t$, an inequality for trigonometric functions may be changed
into a Shafer-Fink type or Carlson type one. For example, by letting $\sin
x=t$, the double inequality (\ref{N<sinx/x<M}) can be changed into%
\begin{equation}
\frac{t}{M\left( \sqrt{1-t^{2}};p,q\right) }<\arcsin t<\frac{t}{N\left( 
\sqrt{1-t^{2}};p,q\right) }  \label{S-F1}
\end{equation}%
if $q\geq 1$ and $p\leq 3q-8/5$ and $\left( p,q\right) \in E_{p,q}$, where $%
M $ and $N$ are defined by (\ref{M}) and (\ref{N}), respectively. Thus
Theorem \ref{MT1 general} can be restated as follows.

\begin{proposition}[Shafer-Fink type inequalities]
\label{P S-F1}Let $t\in (0,1)$ and $\left( p,q\right) \in E_{p,q}$.

(i) If $q\geq 1$ and $p\leq 3q-8/5$, then the doeble inequality (\ref{S-F1})
holds.

(ii) If $34/35<q\leq 1$ and $p\geq \pi ^{2}/4-1$, then the double
inequalities (\ref{S-F1}) is reversed.

(iii) If $0<q\leq 34/35$ and $p\geq 3q-8/5$, then all the double
inequalities (\ref{S-F1}) is reversed.

(iv) If $q\leq 0$ and $p\geq 3q-8/5$, then the inequality%
\begin{equation*}
\arcsin t<\frac{t}{M\left( \sqrt{1-t^{2}};p,q\right) }
\end{equation*}%
holds.
\end{proposition}

Clearly, Proposition \ref{P S-F1} is also true if replacing $\arcsin t$, $t$
and $\sqrt{1-t^{2}}$ with $\arccos t$, $\sqrt{1-t^{2}}$ and $t$,
respectively.

Next we give more refined Shafer-Fink type and Carlson type inequalities.
Employing the monotonicity of $x\mapsto T_{p,q}\left( x/2\right) $ on $%
\left( 0,\pi /2\right) $ given in Proposition \ref{P main1}, we get for $%
pq\neq 0$,%
\begin{equation*}
\frac{1}{3}<T_{p,q}\left( \tfrac{x}{2}\right) =\frac{S_{p}\left( x/2\right) 
}{C_{q}\left( x/2\right) }<\frac{S_{p}\left( \pi /4\right) }{C_{q}\left( \pi
/4\right) }=T_{p,q}\left( \tfrac{\pi }{4}\right)
\end{equation*}%
if $q\geq 1$ and $p\leq 3q-8/5$, that is,%
\begin{eqnarray*}
\frac{1}{3} &<&\frac{q}{p}\frac{1-\left( \frac{2\sin \left( x/2\right) }{x}%
\right) ^{p}}{1-\cos ^{q}\left( x/2\right) }<\frac{q}{p}\frac{1-\left( 2%
\sqrt{2}/\pi \right) ^{p}}{1-2^{-q/2}}=\frac{1}{c_{p,q}}\text{ if }p\neq 0,
\\
\frac{1}{3} &<&q\frac{-\ln \frac{2\sin \left( x/2\right) }{x}}{1-\cos
^{q}\left( x/2\right) }<q\frac{-\ln \left( 2\sqrt{2}/\pi \right) }{1-2^{-q/2}%
}=\frac{1}{c_{0,q}}\text{ if }p=0.
\end{eqnarray*}%
Also, it is easy to check that the largest set of $\{\left( p,q\right) \}$
over the real numbers field such that both the inequalities%
\begin{eqnarray*}
1-\frac{p}{qc_{p,q}}+\frac{p}{qc_{p,q}}\cos ^{q}\left( x/2\right) &>&0\text{
if }q\neq 0, \\
1+\frac{p}{c_{p,0}}\ln (\cos \tfrac{x}{2}) &>&0\text{ if }q=0
\end{eqnarray*}%
hold for all $x\in \left( 0,\pi /2\right) $ is $\mathbb{R}$. While the
largest set of $\{\left( p,q\right) \}$ such that both the inequalities%
\begin{eqnarray*}
1-\frac{p}{3q}+\frac{p}{3q}\cos ^{q}\left( x/2\right) &>&0\text{ if }q\neq 0,
\\
1+\frac{p}{3}\ln (\cos \tfrac{x}{2}) &>&0\text{ if }q=0
\end{eqnarray*}%
hold for all $x\in \left( 0,\pi /2\right) $ is $\{\left( p,q\right)
:p<3q/\left( 1-2^{-q/2}\right) $ if $q\neq 0$ and $p<6/\ln 2$ if $q=0\}$.

Now let $x=\arcsin t$. Then 
\begin{equation*}
\sin \frac{x}{2}=\frac{\sqrt{1+t}-\sqrt{1-t}}{2}\text{, \ }\cos \frac{x}{2}=%
\frac{\sqrt{1+t}+\sqrt{1-t}}{2}.
\end{equation*}%
If let $x=\arccos t$. Then 
\begin{equation*}
\sin \frac{x}{2}=\frac{\sqrt{1-t}}{\sqrt{2}}\text{, \ }\cos \frac{x}{2}=%
\frac{\sqrt{1+t}}{\sqrt{2}}.
\end{equation*}%
And then, by Proposition \ref{P main1} we get respectively

\begin{proposition}[Shafer-Fink type inequalities]
Let $t\in (0,1)$ and $\left( p,q\right) \in \{\left( p,q\right) :p<3q/\left(
1-2^{-q/2}\right) $ if $q\neq 0$ and $p<6/\ln 2$ if $q=0\}$. Then

(i) when $q\geq 1$ and $p\leq 3q-8/5$, the double inequalities%
\begin{equation}
\tfrac{\left( \sqrt{1+t}-\sqrt{1-t}\right) ^{p}}{1-\frac{p}{3q}+\frac{p}{3q}%
2^{-q}\left( \sqrt{1+t}+\sqrt{1-t}\right) ^{q}}<\arcsin ^{p}t<\tfrac{\left( 
\sqrt{1+t}-\sqrt{1-t}\right) ^{p}}{1-\frac{p}{qc_{p,q}}+\frac{p}{qc_{p,q}}%
2^{-q}\left( \sqrt{1+t}+\sqrt{1-t}\right) ^{q}}\text{ if }p>0,
\label{S-F2p>0}
\end{equation}%
\begin{eqnarray}
&&\left( \sqrt{1+t}-\sqrt{1-t}\right) \exp \tfrac{1-2^{-q}\left( \sqrt{1+t}+%
\sqrt{1-t}\right) ^{q}}{3q}  \label{S-F2p=0} \\
&<&\arcsin t<\left( \sqrt{1+t}-\sqrt{1-t}\right) \exp \tfrac{1-2^{-q}\left( 
\sqrt{1+t}+\sqrt{1-t}\right) ^{q}}{qc_{0,q}^{-1}}\text{ if }p=0,  \notag
\end{eqnarray}%
\begin{equation}
\tfrac{\left( \sqrt{1+t}-\sqrt{1-t}\right) ^{p}}{1-\frac{p}{qc_{p,q}}+\frac{p%
}{qc_{p,q}}2^{-q}\left( \sqrt{1+t}+\sqrt{1-t}\right) ^{q}}<\arcsin ^{p}t<%
\tfrac{\left( \sqrt{1+t}-\sqrt{1-t}\right) ^{p}}{1-\frac{p}{3q}+\frac{p}{3q}%
2^{-q}\left( \sqrt{1+t}+\sqrt{1-t}\right) ^{q}}\text{ if }p<0
\label{S-F2p<0}
\end{equation}%
hold, where $3$ and%
\begin{equation}
c_{p,q}=\left\{ 
\begin{array}{ll}
\frac{p}{q}\frac{1-2^{-q/2}}{1-\left( 2\sqrt{2}/\pi \right) ^{p}} & \text{if 
}pq\neq 0, \\ 
\frac{1}{q}\frac{1-2^{-q/2}}{\ln \left( \pi /2\sqrt{2}\right) } & \text{if }%
p=0,q\neq 0, \\ 
\frac{p}{2}\frac{\ln 2}{1-\left( 2\sqrt{2}/\pi \right) ^{p}} & \text{if }%
p\neq 0,q=0, \\ 
\frac{1}{2}\frac{\ln 2}{\ln \left( \pi /2\sqrt{2}\right) } & \text{if }p=q=0%
\end{array}%
\right.  \label{c_p,q}
\end{equation}%
are the best constants;

(ii) when $34/35<q\leq 1$ and $\pi ^{2}/4-1\leq p<3q/\left(
1-2^{-q/2}\right) $, the double inequality (\ref{S-F2p>0}) is reversed;

(iv) when $q\leq 34/35$ and $3q-8/5\leq p<3q/\left( 1-2^{-q/2}\right) $, all
the double inequalities (\ref{S-F2p>0}), (\ref{S-F2p=0}) and (\ref{S-F2p<0})
are reversed.
\end{proposition}

\begin{proposition}[Carlson type inequalities]
Let $t\in (0,1)$ and $\left( p,q\right) \in \{\left( p,q\right) :p<3q/\left(
1-2^{-q/2}\right) $ if $q\neq 0$ and $p<6/\ln 2$ if $q=0\}$. Then

(i) when $q\geq 1$ and $p\leq 3q-8/5$, the double inequalities%
\begin{equation}
\tfrac{\left( 2\sqrt{1-t}\right) ^{p}}{1-\frac{p}{3q}+\frac{p}{3q}%
2^{-q/2}\left( \sqrt{1+t}\right) ^{q}}<\arccos ^{p}t<\tfrac{\left( 2\sqrt{1-t%
}\right) ^{p}}{1-\frac{p}{qc_{p,q}}+\frac{p}{qc_{p,q}}2^{-q/2}\left( \sqrt{%
1+t}\right) ^{q}}\text{ if }p>0,  \label{CTp>0}
\end{equation}%
\begin{eqnarray}
&&\sqrt{2}\sqrt{1-t}\exp \tfrac{1-2^{-q/2}\left( \sqrt{1+t}\right) ^{q}}{3q}
\label{CTp=0} \\
&<&\arccos t<\sqrt{2}\sqrt{1-t}\exp \tfrac{1-2^{-q/2}\left( \sqrt{1+t}%
\right) ^{q}}{qc_{0,q}^{-1}}\text{ if }p=0,  \notag
\end{eqnarray}%
\begin{equation}
\tfrac{\left( 2\sqrt{1-t}\right) ^{p}}{1-\frac{p}{qc_{p,q}}+\frac{p}{qc_{p,q}%
}2^{-q/2}\left( \sqrt{1+t}\right) ^{q}}<\arccos ^{p}t<\tfrac{\left( 2\sqrt{%
1-t}\right) ^{p}}{1-\frac{p}{3q}+\frac{p}{3q}2^{-q/2}\left( \sqrt{1+t}%
\right) ^{q}}\text{ if }p<0  \label{CTp<0}
\end{equation}%
hold, where $3$ and $c_{p,q}$ defined by (\ref{c_p,q}) are the best
constants;

(ii) when $34/35<q\leq 1$ and $\pi ^{2}/4-1\leq p<3q/\left(
1-2^{-q/2}\right) $, the double inequality (\ref{CTp>0}) is reversed;

(iv) when $q\leq 34/35$ and $3q-8/5\leq p<3q/\left( 1-2^{-q/2}\right) $, all
the double inequalities (\ref{CTp>0}), (\ref{CTp=0}) and (\ref{CTp<0}) are
reversed.
\end{proposition}

\subsection{Inequalities for means}

Let $G,A,Q,P,T$ and $U$ stand for the geometric, arithmetic, quadratic, the
first Seiffert \cite{Seiffert.EM.42.1987}, the second Seiffert \cite%
{Seiffert.DW.29.1995} and Yang's means \cite{Yang.JIA.2013.541} of distinct
positive numbers $a$ and $b$ defined by%
\begin{eqnarray*}
G &=&G\left( a,b\right) =\sqrt{ab}\text{, \ }A=A\left( a,b\right) =\frac{a+b%
}{2}\text{, \ }Q=Q\left( a,b\right) =\sqrt{\frac{a^{2}+b^{2}}{2}}, \\
P &=&P\left( a,b\right) =\frac{a-b}{2\arcsin \frac{a-b}{a+b}},T=T\left(
a,b\right) =\frac{a-b}{2\arctan \frac{a-b}{a+b}},U=U\left( a,b\right) =\frac{%
a-b}{\sqrt{2}\arcsin \frac{a-b}{\sqrt{2ab}}},
\end{eqnarray*}%
respectively. The Schwab-Borchardt mean of two numbers $a\geq 0$ and $b>0$,
denoted by $SB(a,b)$, is defined as \cite[Theorem 8.4]{Borwein.PiAGM.1987}, 
\cite[(2.3)]{Brenner.123.1987}, \cite{Neuman.14(2003)}%
\begin{equation*}
SB(a,b)=\left\{ 
\begin{array}{cc}
\frac{\sqrt{b^{2}-a^{2}}}{\arccos (a/b)} & \text{if\ }a<b, \\ 
a & \text{if \ }a=b, \\ 
\frac{\sqrt{a^{2}-b^{2}}}{\func{arccosh}(a/b)} & \text{if \ }a>b.%
\end{array}%
\right.
\end{equation*}%
Very recently, Yang \cite[Theorem 3.1]{Yang.JIA.2013.541} has defined a
family of two-parameter trigonometric sine means as follows.

\begin{definition}
Let $b\geq a>0$ with and $p,q\in \lbrack -2,2]$ such that $0\leq p+q\leq 3$,
and let $\tilde{S}\left( p,q,t\right) $ be defined by%
\begin{equation}
\tilde{S}\left( p,q,t\right) =\left\{ 
\begin{array}{ll}
\left( \dfrac{q}{p}\dfrac{\sin pt}{\sin qt}\right) ^{1/\left( p-q\right) } & 
\text{if }pq\left( p-q\right) \neq 0, \\ 
\left( \dfrac{\sin pt}{pt}\right) ^{1/p} & \text{if }q=0,p\neq 0, \\ 
\left( \dfrac{\sin qt}{qt}\right) ^{1/q} & \text{if }p=0,q\neq 0, \\ 
e^{t\cot pt-1/p} & \text{if }p=q\neq 0, \\ 
1 & \text{if }p=q=0.%
\end{array}%
\right.  \label{S(p,q,t)}
\end{equation}%
Then $\mathcal{S}_{p,q}\left( a,b\right) $ defined by 
\begin{equation}
\mathcal{S}_{p,q}\left( a,b\right) =b\times \tilde{S}\left( p,q,\arccos
\left( a/b\right) \right) \text{ if }a\neq b\text{\ and }\mathcal{S}%
_{p,q}\left( a,a\right) =a  \label{TSmean}
\end{equation}%
is called a two-parameter sine mean of $a$ and $b$.
\end{definition}

As a special case, for $b\geq a>0$, 
\begin{equation*}
\mathcal{S}_{1,0}\left( a,b\right) =b\frac{\sin t}{t}\Big |_{t=\arccos
\left( a/b\right) }
\end{equation*}%
is a mean of $a$ and $b$. Clearly, $\mathcal{S}_{1,0}\left( a,b\right)
=SB\left( a,b\right) $. Thus, after replacing $t$ by $\arccos \left(
a/b\right) $ and multiplying each sides of those inequalities showed in
previous section by $b$, they can be written as corresponding inequalities
for means. For example, Theorem \ref{MT4 p=kq sharp} can be restated as
follows.

\begin{proposition}
\label{PA p=kq sharp}Let $k\in \left( 0,3\right) $ and $b>a>0$. Then

(i) if $k\in \lbrack 7/5,3)$, then the inequality%
\begin{equation}
SB\left( a,b\right) <b^{1-1/k}\left( \left( 1-\frac{k}{3}\right) b^{q}+\frac{%
k}{3}a^{q}\right) ^{1/\left( kq\right) }  \label{SB<}
\end{equation}%
holds if and only if $q\geq 8/\left( 5\left( 3-k\right) \right) $;

(ii) if $k\in \left( p_{0},3\right) $, then the inequality%
\begin{equation}
b^{1-1/k}\left( \left( 1-\frac{k}{3}\right) b^{q}+\frac{k}{3}a^{q}\right)
^{1/\left( kq\right) }<SB\left( a,b\right)  \label{SB>}
\end{equation}%
holds if and only if $q\leq \left( \ln \left( 1-k/3\right) \right) /\left(
k\ln \left( 2/\pi \right) \right) $, where $p_{0}\approx 1.42034$ is defined
by (\ref{p_0,p_0*});

(iii) if $k\in (0,23/17]$, then the inequality (\ref{SB>}) holds if and only
if $q\leq 8/\left( 5\left( 3-k\right) \right) $;

(iv) if $k\in \left( 0,p_{0}^{\ast }\right) $, then the inequality (\ref{SB<}%
) holds if and only if $q\geq \left( \ln \left( 1-k/3\right) \right) /\left(
k\ln \left( 2/\pi \right) \right) $, where $p_{0}^{\ast }$ $\approx 1.27754$
is defined by (\ref{p_0,p_0*}).
\end{proposition}

As another example, Theorem \ref{MT5 p=3q-8/5b} can be restated in the
following form.

\begin{proposition}
\label{PA p=3q-8/5 sharp}Let $b>a>0$. Then the inequality $SB\left(
a,b\right) >b\times M\left( a/b;3q-8/5,q\right) $ holds if and only if $%
q\leq 34/35$, while its reverse holds if and only if $q\geq q_{0}\approx
0.989681$, where $q_{0}$ is the unique root of equation $D_{3q-8/5,q}\left(
\pi /2^{-}\right) =0$ on $\mathbb{R}_{+}$, here $D_{3q-8/5,q}\left( \pi
/2^{-}\right) $ is defined by (\ref{D_1q-8/5,q}).
\end{proposition}

Further, let $m=m(a,b)$ and $M=M(a,b)$ be two means of $a$ and $b$ with $%
m(a,b)<M(a,b)$ for all $a,b>0$. Clearly, making a change of variables $%
a\rightarrow m\left( a,b\right) $ and $b\rightarrow M(a,b)$, $\mathcal{S}%
_{p,q}\left( m,M\right) $ is still a mean of $a$ and $b$ which lie in $m$
and $M$. Particularly, taking $\left( m,M\right) =\left( G,A\right) $, $%
\left( A,Q\right) $, $\left( G,Q\right) $, respectively, we can obtain new
symmetric means as follows:%
\begin{eqnarray*}
\mathcal{S}_{1,0}\left( G,A\right) &=&SB\left( G,A\right) =\frac{a-b}{%
2\arcsin \frac{a-b}{a+b}}=P\left( a,b\right) , \\
\mathcal{S}_{1,0}\left( A,Q\right) &=&SB\left( A,Q\right) =\frac{a-b}{%
2\arctan \frac{a-b}{a+b}}=T\left( a,b\right) , \\
\mathcal{S}_{1,0}\left( G,Q\right) &=&SB\left( G,Q\right) =\frac{a-b}{\sqrt{2%
}\arcsin \frac{a-b}{\sqrt{2ab}}}=U\left( a,b\right) ,
\end{eqnarray*}%
where $P,T,U$ are the first Seiffert, the second Seiffert, Yang's means.

And then, after letting $t=\arccos \left( a/b\right) $ and multiplying each
sides of those inequalities showed in previous section by $b$, and replacing 
$\left( a,b,SB\left( a,b\right) \right) $ with $\left( G,A,P\right) $, $%
\left( A,Q,T\right) $, $\left( G,Q,U\right) $, we can establish
corresponding inequalities for symmetric means. For example, Theorem \ref%
{MT5 p=3q-8/5b} can be restated in the following form.

\begin{proposition}
Let $a,b>0$ with $a\neq b$. Then the inequality%
\begin{equation*}
P>\left\{ 
\begin{array}{ll}
A^{\left( 10q-8\right) /\left( 15q-8\right) }\left( \frac{8}{15q}%
A^{q}+\left( 1-\frac{8}{15q}\right) G^{q}\right) ^{5/\left( 15q-8\right) } & 
\text{if }q\neq 0,8/15, \\ 
A\left( 1-\frac{8}{15}\ln \frac{G}{A}\right) ^{-5/8} & \text{if }q=0, \\ 
A\exp \frac{5\left( \left( G/A\right) ^{8/15}-1\right) }{8} & \text{if }%
q=8/15%
\end{array}%
\right.
\end{equation*}%
hold if and only if $q\leq 34/35$, while their reverse hold if and only if $%
q\geq q_{0}\approx 0.989681$, where $q_{0}$ is the unique root of equation $%
D_{3q-8/5,q}\left( \pi /2^{-}\right) =0$ on $\mathbb{R}_{+}$, here $%
D_{3q-8/5,q}\left( \pi /2^{-}\right) $ is defined by (\ref{D_1q-8/5,q}).
\end{proposition}

\begin{remark}
When replacing $\left( G,A,P\right) $ by $\left( A,Q,T\right) $ and $\left(
G,Q,U\right) $, the above proposition are still valid.
\end{remark}

\end{document}